\documentclass[10pt]{article}
\usepackage{amsmath,amssymb,amsthm,amscd}

\begin{document}

\title{K\"ahler-Einstein metrics on Fano manifolds, III: limits as cone angle approaches $2\pi$ and completion of the main proof. }
\author{Xiuxiong Chen, Simon Donaldson and Song Sun\footnote{X-X. Chen was partly supported by National Science Foundation grant No 1211652; S. Donaldson and S. Sun were partly supported by the European Research Council award No 247331.}}\date{\today}
\maketitle



\def\ZZ{{\mathbb Z}}
\def\QQ{{\mathbb Q}}
\def\RR{{\mathbb R}}
\def\NN{{\mathbb N}}

\def\p{\partial}
\def \bp {\overline{\partial}}
\def \H {\mathcal H}
\def \Di {\mathcal D}
\def \dVol {\text dVol}

\newtheorem{thm}{Theorem}
\newtheorem{theo}{Theorem}
\newtheorem{prop}{Proposition}
\newtheorem{lem}{Lemma}
\newtheorem{cor}{Corollary}
\newtheorem{defn}{Definition}
\newtheorem{rem}{Remark}
\newcommand{\tomega}{\tilde{\omega}}
\newcommand{\tD}{\tilde{D}}
\newcommand{\Ric}{{\rm Ricci}}
\newcommand{\bC}{{\bf C}}
\newcommand{\bQ}{{\bf Q}}
\newcommand{\bP}{{\bf P}}
\newcommand{\bZ}{{\bf Z}}
\newcommand{\bR}{{\bf R}}
\newcommand{\db}{\overline{\partial}}
\newcommand{\cD}{{\cal D}}
\newcommand{\uf}{\underline{f}}
\newcommand{\us}{\underline{s}}
\newcommand{\uv}{\underline{v}}
\newcommand{\dbd}{i \partial \overline{\partial}}
\newcommand{\Euc}{{\rm Euc}}
\newcommand{\tr}{{\rm Tr}}
\newcommand{\Aut}{{\rm Aut}}
\newcommand{\Fut}{{\rm Fut}}

\section{Introduction}

This is the third and final paper in a series which establish results announced in \cite{CDS0}. We will begin by recalling the central result and some background.
Suppose given an $n$-dimensional  compact K\"ahler  manifold $X$ the general question, going back to Calabi \cite{Calabi} is to ask whether $X$ admits a {\it K\"ahler-Einstein metric}: a K\"ahler metric $\omega$ whose Ricci form $\rho$ is a multiple of $\omega$. Since the Ricci form represents the characteristic class $c_{1}(X)$ a necessary condition is that $c_{1}$ is either positive, negative or zero (in real cohomology). In renowned work in the 1970's, the negative case was settled by Aubin \cite{Au} and Yau \cite{Ya}, and the case when $c_{1}=0$---the \lq\lq Calabi-Yau case''---by  Yau \cite{Ya}. It has been known for many years that in the positive case, when $X$ is a \lq\lq Fano manifold'',
new features arise in that there are obstructions to solving the K\"ahler-Einstein equation and hence no straightforward general existence  theorem, of the same kind as those for $c_{1}< 0$ or $c_1=0$. The first such obstruction was found by Matsushima \cite{Matsushima}, who showed that if a Fano manifold $X$ admits a K\"ahler-Einstein metric then the holomorphic automorphism group of $X$ is reductive: thus, for example, the blow-up of the projective plane at one point does not carry a K\"ahler-Einstein metric.  Later, Futaki \cite{Fut} discovered  a character of the automorphism group (the {\it Futaki invariant}) which depends only on the complex geometry of $X$ but which must vanish if $X$ admits a K\"ahler-Einstein metric. So, for example, one can easily write down many toric Fano manifolds which do not have such metrics. (In fact Futaki's theory applies in the wider world of constant scalar curvature K\"ahler metrics \cite{Calabi2}.)

There has been an immense amount of work (which we will not attempt to summarise) establishing existence theorems in particular cases. Notably, Tian gave a complete solution of the problem in complex dimension 2 \cite{Ti0}. The idea that the correct criterion for the existence of a solution should involve the algebro-geometric notion of {\it stability} goes back to Yau \cite{kn:Y}. In one direction, Tian established the necessity of a stability condition and demonstrated by an example that this gave a condition going beyond those previously known, in that there are Fano manifolds with trivial automorphism group which do not admit K\"ahler-Einstein metrics \cite{Ti1}. But the other direction, to establish existence a stability condition, has proved much harder and this is what we achieve here. A variety of stability conditions have been considered in the literature but the one we use is essentially the {\it K-stability} introduced by Tian in \cite{Ti1}. We will not recall the definition in detail just here but postpone that to a technical discussion in Section 5 below. In brief, testing the stability of $X$ involves considering equivariant degenerations over the disc with general fibre $X$ but some possibly different central fibre $X_{0}$, which has a $\bC^{*}$-action. This central fibre is allowed to be singular and  specifying precisely what singularities are allowed is a part of the technical discussion. The stability condition involves the Futaki invariant of this $\bC^{*}$ action on $X_{0}$, and the extension of that theory to the singular situation is another important part of the technical discussion.
Assuming,  then the definition of K-stability as known, our main result (announced in \cite{CDS0}) is
 \begin{thm} \label{main theo}
 If a  Fano manifold $X$ is $K$-stable then it admits  a K\"ahler-Einstein metric.
\end{thm}

As indicated above, the converse is also known (see further references in Section 5 below). Thus K-stability is a necessary and sufficient condition for the existence of a K\"ahler-Einstein metric.
The basic point about this result is that K-stability is an entirely algebro-geometric condition, so we have in principle an algebro-geometric criterion for the existence of a K\"ahler-Einstein metric. On the other hand we should point out that as things stand at present the result is of very limited use in concrete cases, so that there is no manifold $X$ known to us, not covered by other existence results and where we can deduce that $X$ has a  K\"ahler-Einstein metric. This is because it seems a very difficult matter to test K-stability by a direct study of all possible degenerations. However we are  optimistic that this situation will change in the future, with a deeper analysis of the stability condition. As Yau has pointed out \cite{kn:Y}, our result has an interesting theoretical consequence.  The tangent bundle of a K\"ahler-Einstein manifold has a Hermitian Yang-Mills metric and (assuming that the manifold is not a product) is thus a stable bundle  in the sense of Mumford. So we have the algebro-geometric statement that the K-stability of a Fano manifold implies the stability of its tangent bundle.

We now move to a more technical level in this Introduction and outline the detailed results proved in this paper. In Section 5 we return to the larger picture and explain how these results, along with those in \cite{CDS1}, \cite{CDS2} lead to a proof of our main result.

To avoid repetition we assume that the reader is familiar with the introduction of \cite{CDS2} and we adopt the same notation, so we have a sequence of Fano manifolds $X_{i}$ with K\"ahler-Einstein metrics $\omega_{i}$ having cone singularities along $D_{i}$ with cone angle $2\pi\beta_{i}$ where $\beta_{i}\in (0,1)$ converges to $\beta_{\infty}$. The divisors $D_{i}$ are smooth divisors in  $\vert- \lambda K_{X_i}\vert$ for some fixed $\lambda$. Our previous paper studied  the case when $\beta_{\infty}<1$ and the first point of the present paper is to extend the results to the case when $\beta_{\infty}=1$.
 Thus we prove
 
 \begin{thm} \label{thm1}
Let $X_{i}$ be a sequence of  Fano manifolds with fixed Hilbert polynomial. Let $D_{i}\subset X_{i}$ be smooth divisors in $\vert -\lambda K_{X_{i}}\vert $, for fixed $\lambda> 0$. Let $\beta_{i}\in (0,1)$ be a sequence  converging to $1$.  Suppose that there are K\"ahler-Einstein metrics $\omega_{i}$ on $X_{i}$ with cone angle $2\pi\beta_{i}$ along $D_{i}$. Then there is a $\bQ$-Fano variety $W$  such that 
\begin{enumerate} \item there is a weak  K\"ahler-Einstein metric $\omega$ on $W$ (In particular, by our definition in \cite{CDS2} if $W$ is smooth, then $\omega$ is a genuine smooth K\"ahler-Einstein metric);
\item  possibly after passing to a subsequence, there are embeddings $T_{i}:X_{i}\rightarrow \bC\bP^{N}, T_{\infty}:W\rightarrow \bC\bP^{N}$, defined by the complete linear systems $\vert -m K_{X_{i}}\vert$ and $\vert -m K_{W}\vert$ respectively (for a suitable $m$ depending only on the dimension of $X_i$ and $\lambda$), such that $T_{i}(X_{i})$ converge to $T_{\infty}(W)$.
\end{enumerate}
\end{thm}
This is the exact analogue of Theorem \ref{thm1} in \cite{CDS2}, with the difference that the cone singularity \lq\lq disappears'' in the limit when $\beta=1$. (Although it is not stated explicitly in \cite{CDS2}, it is clear that the power $m$ in that case also depends only on the dimension of $X_i$, $\lambda$ and the positive lower bound of Ricci curvature.) The precise definition of a weak K\"ahler-Einstein metric on a $\bQ$-Fano variety is given in the introduction to \cite{CDS2}.

As in \cite{CDS2}, the proof of Theorem \ref{thm1} goes via a study of the Gromov-Hausdorff limit. We may suppose  by taking a subsequence and using Theorem 1 or 2 in \cite{CDS1} that the $(X_{i},\omega_{i})$  do have such a limit $Z$.
Then as explained in \cite{CDS2} the Cheeger-Colding-Tian structure theory applies to $Z$, so $Z$ is a disjoint union ${\cal R}\cup {\cal S}$ where ${\cal S}={\cal S}_{1}\cup {\cal S}_{2}\cup \dots$ and ${\cal S}_{i}$ has Hausdorff codimension at least $2i, $ while
$\cal R$ is the set of regular points. Our main result on the structure of the limit is 

\begin{thm} \label{thm3}
The set ${\cal R}$ is open and the metric on ${\cal R}$ is a smooth K\"ahler Einstein metric $\omega$. We can choose local complex coordinates so that $\omega_{i}$ converges to $\omega$ in $L^{p}$ for all $p$. The same applies to any iterated tangent cone to $Z$.
\end{thm}

From this we are able to prove:
\begin{cor} \label{cor1}
 ${\cal S}={\cal S}_{2}$. The same applies to any iterated tangent cone of $Z$.
\end{cor}

In other words we have established the same results about the structure of Z as are known, in the standard theory, for limits of smooth K\"ahler-Einstein manifolds. 

Given these results it is straightforward to adapt the arguments in \cite{kn:DS}, \cite{CDS2} to prove Theorems 1. See section 2  below. We have two approaches to prove Theorem \ref{thm3}. The first uses similar techniques to those in \cite{CDS2}, but there are new difficulties to overcome. This is done in Section 2 below. We give the second proof in Section 3: this uses Ricci flow  and a variant of Perelman's pseudolocality theorem due to Tian and Wang \cite{WT12}.    In the first approach, we only need to approximate the metrics $\omega_i$   by smooth K\"ahler metrics with Ricci curvature bounded from below. In the second approach, we need to use the full strength of our approximation result \cite{CDS1} that $(X_{i}, \omega_{i})$ can
be approximated by metrics with Ricci curvature bounded below by $\beta_{i}\;$ as $\beta_i \rightarrow 1.\;$\\

 In Section 4 we consider the automorphism group of a pair $(W,\Delta)$ and we prove the relevant statements in \cite{CDS0}. 
\begin{thm} \label{thm4} Let $W$ be a $\bQ$-Fano variety, $\Delta$ a Weil divisor and $0<\beta<1$. If there is a weak conical K\"ahler-Einstein metric for the triple $(W,\Delta, \beta)$ then the automorphism group of $(W,\Delta)$ is reductive and the Futaki invariant is zero.
\end{thm}
 In this statement we include the case when $\Delta$ is empty, so the metric is a
weak K\"ahler-Einstein metric on $W$, as above. Of course in the case when $\Delta$ is empty and $W$ is smooth this result comes down to the well-known results of Matsushima and Futaki mentioned above, so the point is to deal with the singularities.

As stated before, in Section 5 we put all our technical results together to give the proof of the main theorem. In the Appendix 1 we give an exposition of the proof of convexity of the Ding functional and the uniqueness of weak conical K\"ahler-Einstein metrics which is used in the proof of Theorem \ref{thm4}. In Appendix 2 we likewise give an exposition of the proof of a PDE result of Evans-Krylov type, which we use in Section 2.  \\

{\bf Acknowledgements:} The first author would like to express his deep gratitude to
his parents, who put  all the little they had to give him the best education,
and to his wife,  without
whose  love, support, and continued sacrifice he would never have had
the courage to persevere in a mathematical career.

       We thank Weiyong He,  Long Li, Bing Wang and Yuanqi Wang for help during the preparation of this series of papers.
    \section{First proof of Theorem \ref{thm3} and  proof of Corollary \ref{cor1}}
    \subsection{Overview }
    The focus in this section is on a Riemannian ball $B=B(p,1)$ (in a K\"ahler-Einstein manifold $(X,\omega)$ with cone singularities along a divisor $D$) which is Gromov-Hausdorff close to the unit ball $B^{2n}$. Here we are considering the scaled case but we assume the scaling is by a factor $a^{1/2}$ for integer $a$ so there is a line bundle $L\rightarrow X$ with curvature $\omega$. We also suppose that we have a fixed distance function on $B\sqcup B^{2n}$ as in the definition of Gromov-Hausdorff distance. The fundamental difficulty, when the cone angle tends to $2\pi$, is in controlling the singular set $D\cap B$. To illustrate this consider the local discussion in the case of complex dimension $1$. Then any metric on the ball can be Gromov-Hausdorff  approximated by a flat metric with cone singularities of cone angle close to $2\pi$ at many points. So without control of the number of points we cannot say anything more about the metric. In general we will exert control  by arguments based on a generalisation of the notion of \lq\lq Minkoswki measure''. 
    
    Recall that for any set $A$ in a $2n$-dimensional length space $P$ (of a suitable kind) we define the  Minkowski measure $m(A)$ as in Section 2 in \cite{CDS2}. We want to extend this notion slightly. Given $\eta> 0$ we let $m(\eta, A)$ be the infimum of numbers $M$ such that for all $r\geq \eta$ there is a cover of $A$ by at most $Mr^{2-2n}$ balls of radius $r$ (with $m(\eta,A)=\infty$ if there are no such numbers $M$). Thus $m(A)\leq M$ if and only if $m(\eta, A)\leq M$ for all $\eta>0$. \label{minkowskimeasure}

We want an appropriate notion of \lq\lq good local holomorphic co-ordinates''. Fix $p>2n$ (in order to have a Sobolev embedding $L^{p}_{1}\rightarrow C^{0}$). Given $x\in B$ and numbers $r,\delta'>0$ we say that a holomorphic map $\Gamma:B(x,r)\rightarrow
\bC^{n}$ is an $(r,\delta')$-chart centred at $x$ if 
\begin{enumerate}
\item $\Gamma(x)=0$;
\item $\Gamma$ is a homeomorphism to its image;
\item For all $x',x''\in B_{x}(r)$ we have $\vert d(x',x'')-d(\Gamma(x'), \Gamma(x''))\vert\leq \delta'$;
\item $$\Vert \Gamma_{*}(\omega)-\omega_{\Euc}\Vert_{L^{p}}\leq \delta'.$$
\end{enumerate}

Next we want to introduce a useful invariant $I(\Omega)$ for any domain $\Omega\subset X$. Recall that for $x\in X$ and $r>0$ we define the volume ratio
$ VR(x,r)$ to be the ratio of the volume of the ball $B(x,r)$ and the Euclidean ball $r B^{2n}$. Then we set 
\begin{equation}\label{isoperi}
  I(\Omega)= \inf_{B(x,r)\subset \Omega} VR(x,r). 
  \end{equation}
In our situation, of positive Ricci curvature, the Bishop inequality states that $VR(x,r)\leq 1$ so $I(\Omega)\leq 1$.  We will only be concerned with the case when $\Omega$ is a ball, in particular when $\Omega=B$. Then, by fundamental results of Cheeger and Colding, $1-I(B)$ is closely related to the Gromov-Hausdorff distance $d_{GH}(B, B^{2n})$. That is, for any $\lambda>1$,
$1-I(B)$ is controlled by $d_{GH}(B(p,\lambda), \lambda B^{2n})$ and
 $d_{GH}(B, B^{2n})$  is controlled by $1-I(B)$.  The advantage of the invariant $I(\Omega)$ is that, from the definition, it has a monotonicity property $I(\Omega')\geq I(\Omega)$ for $\Omega'\subset \Omega$. 

\

We can now state the main technical result of this Section.
\begin{prop}\label{Hormander}
 Given $M,c$ there are $\rho(M), \eta(M,c),\delta(M,c)>0$ with the following effect. Suppose $1-I(B)\leq \delta$ and  that $W\subset B$ is a subset with $m(\eta,W)\leq M$ such that for any point $x$ in $B\setminus W$ there is a $(c\eta, \delta)$-chart centred at $x$. Then (for some fixed $C$, independent of $M,c$)
\begin{itemize} \item There is a local K\"ahler potential $\phi$ for $\omega$ on the ball $B(p,\rho)$ with $\vert \phi\vert \rho^{-2}\leq C$. 
\item There is a holomorphic map $F$ from $B(p,\rho)$ to  $\bC^{n}$  which is a homeomorphism to its image, satisfies a Lipschitz bound $\vert \nabla F\vert \leq  C$ and such that the image lies between 
 $(0.9) \rho B^{2n}$ and $(1.1)\rho B^{2n}$.
 \end{itemize}
 \end{prop}
 
Stated more loosely and informally: the good local co-ordinates at points outside the \lq\lq small'' set $W$ go over to a controlled local co-ordinate on a  ball of definite size $\rho$ around $p$, where the size depends on $M$. In fact we will only need to apply this Proposition with some fixed computable number $M$, so effectively we can think of $\rho$ as being fixed.

We prove Proposition \ref{Hormander} using the H\"ormander technique, following the same general strategy as in \cite{kn:DS}, \cite{CDS2}. Thus we want to have a suitable \lq\lq large'' open set $\Omega$ in the complement of $W$ on which the geometric structures can be identified, up to small errors, with the flat model on a corresponding set in $B^{2n}$. We want to have a cut-off function with derivative small in $L^{2}$. We transport holomorphic sections of the line bundle from the model to get approximately holomorphic sections over $B$ and project these to genuine holomorphic sections. Then the proof of Proposition \ref{Hormander} is much the same as the proof of Proposition 12 in \cite{CDS2}.
All of this is covered in subsections 2.2, 2.3, 2.4 below. 

We apply Proposition \ref{Hormander} twice in the proof of Theorem 2. The first application is relatively straightforward and a simpler statement would suffice. This leads to a lower bound on the local densities of the divisor (Corollary \ref{cor2}).  The second application is more involved and uses the full force of the statement. This is used in an inductive argument to handle the situation, with some {\it fixed} $\delta$, where the volume of $D\cap B$ is possibly large but where the cone angle is sufficiently close to $2\pi$.

    \subsection{Holonomy}
    
    One of the  main technical difficulties in the proof of Proposition \ref{Hormander} involves holonomy. Thus we begin with a treatment of that. We state and prove a rather involved technical result (Proposition \ref{prop1} below)  which will be just what we need later.

For a  subset $W\subset B^{2n}$ and $s>0$ we define
$\Omega(W,s)\subset B^{2n}$ to be set of points of distance greater than $s$ from any point of $W$ and also from the boundary of $B^{2n}$. Given $C>1,\eta>0$ we say that a Euclidean circle of radius between $\eta$ and $C\eta$ and contained in $\Omega(W, \eta)$ is  a {\it standard $(C,\eta)$ circle for $W$}.

\begin{prop} \label{prop1}
Given $r,M,\theta>0$ there exists $\eta_{0}, \zeta, C>0$ such that if $\eta\leq \eta_{0}$ there is a $\psi>0$ (depending on $\eta$) with the following effect. If $W\subset B^{2n}$ is a set with $m(\eta, W)\leq M$ and if $a$ is a connection on an $S^{1}$ bundle  over $\Omega(W,\eta)$  whose holonomy around every standard $(C,\eta)$-circle for $W$ is bounded by $\zeta$ and with $L^{p}$ norm of the curvature  less than $\psi$, then over $\Omega(W,r)$  there is a trivialisation of the bundle so that connection form has $\vert a\vert \leq \theta$.
\end{prop}

(Here, in saying that the holonomy is bounded by $\zeta$ we mean that the holonomy is $\exp(i\alpha)$ for $\vert \alpha \vert \leq \zeta$.)

The first part of the proof  of  Proposition 1 uses a compactness argument.  

Suppose the statement is false, for some fixed $r,M,\theta$. Then we have $\eta_{i}, \zeta_{i}\rightarrow 0$ and $C_{i}\rightarrow \infty$ such that we can find the following

\begin{itemize}
\item For each $i$ a sequence $\psi_{ij}>0$ tending to zero as $j\rightarrow \infty$.
 \item Sets $W_{ij}$ with $m(\eta_{i}, W_{ij})\leq M$.
\item  connections $A_{ij}$ over $\Omega(W_{ij}, \eta_{i})$ such that the holonomy around each standard $(C_{i}, \eta_{i})$ circle for $W_{ij}$ is bounded by $\zeta_{i}$, and the $L^{p}$ norm of the curvature is bounded by $\psi_{ij}$;
\item there is no trivialisation of the bundle over $\Omega(W_{ij}, r)$ in which the connection form is bounded by $\theta$.
\end{itemize}
 
 Fix $i$ and choose a cover of each set $W_{ij}$ by a fixed number of $\eta_{i}$ balls, as in the definition of $m(\eta, \cdot)$. Taking a subsequence we can suppose the centres of these balls converge as $j\rightarrow \infty$. Let $U_{i}$ be the union of the $2\eta_{i}$ balls centred on the limit points.
Then one sees from the definition that $m(\eta_{i}, U_{i}) \leq M'$ for some fixed $M'$ depending on $M,n$. Again taking a subsequence, we can suppose that the $A_{ij}$ converge in $C^{0}$ on compact subsets of $B^{2n}\setminus U_{i}$ to a flat connection $A_{i}$. That is, we use a local gauge fixing to get local convergence in $L^{p}_{1}$
and the Sobolev embedding gives local $C^{0}$ convergence. Then we can glue together the local trivialisations to get $C^{0}$ convergence on compact sets, as in \cite{kn:Uhlenbeck}. (But the abelian case here is much more straightforward.) The definitions imply that the holonomy of $A_{i}$ around any standard $(C_{i}/2, 2 \eta_{i})$ circle for $U_{i}$ is bounded by $\zeta_{i}$.
We can suppose that $\eta_{i}$ is a decreasing sequence so $m(\eta_{i}, U_{i'})\leq M'$ for all $i'\geq i$. For each $i$  for all integers $t$ such that $2^{-t}\geq \eta_{i}$,  and for all $i'\geq i$ choose a cover of $U_{i'}$ by a fixed number of $2^{-t}$  balls, as in the definition of $m(\eta_{i}, \cdot )$. Taking a subsequence we can suppose that for all $t$ the centres of these balls converge as $i'\rightarrow \infty$. Let $V_{i,t}$ be the union of the closed $2^{-t}$ balls with these limiting centres and set $W_{\infty}=\bigcap_{i,t} V_{i,t}$. Then $W_\infty$ is a closed set, $m(W_\infty)\leq M''$ for some fixed $M''$ depending on $M,n$ and if $p$ is a point with $d(p,W_\infty)> s$ for some $s$ then $d(p,U_{i})\geq s$ for large $i$. We get a limiting flat connection $A_{\infty}$ over $B^{2n}\setminus W_{\infty}$, where the limit is in the sense of convergence on compact subsets. If we know that $A_{\infty}$ is the trivial flat connection then we get a contradiction to our hypotheses and the Proposition is proved.

We now pass to the second part of the proof.
\begin{lem} \label{lem1}
Let $W_{\infty}\subset B^{2n}$ be a closed set with $m(W_{\infty})<\infty$ and let $\Gamma$ be a loop in $B^{2n}\setminus W_{\infty}$. Then we can find
$N,C>0$ such that for all sufficiently small $\eta$ the loop $\Gamma$ is contained in $ \Omega(W_{\infty}, \eta)$ and  is homologous in $\Omega(W_{\infty}, \eta)$ to a sum of at most $N$ standard $(C,\eta)$ circles for $W_{\infty}$.
\end{lem}
Assuming this Lemma for the moment we complete the proof of the Proposition. Suppose, arguing for a  contradiction that the connection $A_{\infty}$ is not trivial, so has nontrivial holonomy around some loop $\Gamma$. Let this holonomy be $e^{i\gamma}$. Choose $\zeta= \vert \gamma\vert/(10 N)$ where $N$ is as in the Lemma.  Choose $i$ large enough that $\zeta_{i}<\zeta$. The standard $(C,\eta)$-loops for $W_\infty$ furnished by the Lemma are standard $(2C,\eta/2)$ loops for $U_{i}$ for large enough $i$, so the holonomy around each is bounded by $\zeta_{i}\leq \zeta$ once $2\eta_{i}<\eta/2$ and $C_{i}> 4C$. This gives our contradiction.

 We move on to prove the Lemma.
 We can assume that $\Gamma$ is an embedded piecewise linear loop, determined by a sequence of vertices $q_{1}, \dots, q_{p}=q_{0}$. There is clearly no real loss of generality is supposing that $0$ does not lie in $W_{\infty}$ or $\Gamma$. Likewise we can make the general position assumption that the triples $0 q_{j} q_{j+1}$ span distinct planes $\Pi_{j}$ and that $\Pi_{j}$ meets $\Pi_{k}$ only at $0$ if $k\neq j\pm 1$. Choose a real $(2n-2)$ dimensional subspace $L\subset \bC^{n+1}$ transverse to all the planes $\Pi_{j}$. For each $v \in L$ we can take the cone $C(v,\Gamma)$ over $\Gamma$ with vertex $v$. Thus $C(v,\Gamma)$  is a union of triangles $\Delta_{j,v}$ with vertices $v, q_{j}, g_{j+1}$. 

Consider the projection from $\bC^{n}\setminus \{0\}$ to the unit sphere $S^{2n-1}$. The image of $W_{\infty}$ has finite Minkowski measure and so in particular has empty interior. This means that, after perhaps making a small perturbation to $\Gamma$ we can suppose that $W_{\infty}$ does not meet the rays $0,q_{j}$. It follows that for sufficiently small $v$ the intersection of $W_{\infty}$ and $C(v,\Gamma)$ lies strictly inside the interior of the triangles $\Delta_{j,v}$. Thus we can find $\lambda_{1}, \lambda_{2}>0$ such that if $\vert v\vert \leq \lambda_{1}$ the distance between two points in $W_{\infty}\cap C(v,\Gamma)$ which lie in different triangles exceeeds $\lambda_{2}$.
We can then find $\kappa_{1}>0$ such that if $\vert v_{1}\vert , \vert v_{2}\vert \leq \lambda_{1}$ and if $w_{i}\in C(v_{i}, \Gamma)\cap W_{\infty}$ then 
$\vert w_{1}-w_{2}\vert \geq \kappa_{1} \vert v_{1}-v_{2}\vert$.

Given $\eta>0$ we take a finite set of $S$ points $v_{i}$ in the ball of radius $\lambda_{1}$ in $L$ such that the distance between any two points exceeds $20 \kappa_{1}^{-1} \eta$ and with $S\geq \kappa_{2} \eta^{2-2n}$, for fixed $\kappa_{2}>0$. Using the definition of Minkowski measure we can find a set of $T$ points $z_{\nu}$ in $W_{\infty}$ so that the $\eta$ balls with these centres cover $W_{\infty}$, where $T\leq M \eta^{2-2n}$. 
 Suppose that $x_{i}\in C(v_{i}, \Gamma)$ and $x_{j}\in C(v_{j},\Gamma)$ for distinct $i,j$. Suppose also that $\vert x_{i}- z_{\mu}\vert \leq 2\eta$ and $\vert x_{j}- z_{\nu}\vert\leq 2\eta$ for some $\mu,\nu$.
 Then we cannot have $\mu=\nu$, since then $\vert x_{i}-x_{j}\vert \leq 4\eta$.
 For each $i$ let $M_{i}$ be the number of points $z_{\nu}$ whose $2\eta$ neighbourhoods intersect $C(v_{i},\Gamma)$. A counting argument shows that we can find a fixed $K_{0}$ (independent of $\eta$) so that for some $i$ we have $M_{i}\leq K_{0}$. Having found this vector $v_{i}$ we just write $v_{i}=v$. Let $U$ be the intersection of $C(v,\Gamma)$ with  all the $2\eta$ balls centre $z_{\nu}$. Thus $U$ is the union of at most $K_{0}$ discs in the interiors of the triangles $\Delta_{i,v}\subset C(v,\Gamma)$. On the other hand $U$ contains the intersection of  $C(v,\Gamma)$ with the $\eta$ neighbourhood of $W_{\infty}$. The proof of the Lemma follows easily. This completes the proof of Proposition 1.

\subsection{Matching with the flat model}

To begin we state:
    \begin{lem}\label{lemcon} Given $r,M>0$ there is an $\eta_{1}>0$ such that if  $W\subset B^{2n}$ is a set with $m(\eta_{1}, W)\leq M$ then any two points in $\Omega(W,r)$ can be joined by a path in $\Omega(W,\eta)$.
\end{lem}

The proof  is similar to the proof of  Lemma \ref{lem1} above, but simpler. Given points $x,y\in \Omega(W,r)$ we consider a $(2n-1)$-dimensional family of piecewise linear paths from $x$ to $y$. Then a counting argument shows that when $\eta$ is sufficiently small at least one of these paths must lie in $\Omega(W,\eta)$. We leave the reader to fill in details. 

\

\begin{lem}\label{lemmap}
Let $U$ be a domain in $\bC^{n}$ and let $a,b$ be points in $\bC^{n}$. Let $f_{i}:U\rightarrow \bC^{n}$ be a sequence of maps which can be written in the form $f_{i}= g_{i} \tilde{f}_{i} h_{i}^{-1} $ where $\tilde{f}_{i}$ are holomorphic and  $g_{i}, h_{i}$ are Euclidean transformations fixing $a,b$ respectively.  If $f_{i}$ converges to the inclusion map $U\rightarrow \bC^{n}$  in $C^{0}$ then they also converge in $C^{\infty}$ on compact subsets of $U$. \end{lem}

This is clear from  the compactness of the orthogonal group and elliptic estimates for holomorphic maps.

\

\begin{prop} \label{prop2}
For  $\tilde M,c, r ,\tilde{\theta}>0$, there are $\delta, \eta$ such that the following holds. If $1-I(B)\leq \delta$ and $W\subset B$ is a subset with $m(\eta, W)\leq \tilde M$ and for any point $x$ not in $W$ there is a $(c\eta,\delta)$-chart centred at $x$ then there is an open embedding $\phi: \Omega (W, r) \rightarrow B^{2n}$ such that $\Vert \phi_{*}(\omega)-\omega_{\Euc}\Vert_{L^{p}}\leq \tilde\theta, \vert \phi_{*}(J)-J_{\Euc}\vert \leq \tilde{\theta}$ and $d(x,\phi(x))\leq \tilde{\theta}$. 
\end{prop}
Here $J,\omega$ are the complex structure and 2-form on $B$ and $ J_{\Euc}, \omega_{\Euc}$ the standard structures on $B^{2n}$.  

 Recall that $1-I(B)$ controls the Gromov-Hausdorff distance of $B$ to $B^{2n}$. To streamline notation let us write $\delta^{*}(\delta)$ for any function that tends to $0$ with $\delta$, which can vary from line to line. Thus for example we have
$1-I(B)\leq \delta$ implies that $d_{GH}(B, B^{2n})\leq \delta^{*}$.
 
First we use Lemma \ref{lemcon} to fix $\eta$ so that for a subset $W\subset B^{2n}$ with $m(\eta,W)\leq \tilde M$ any two points in $\Omega(W,r)$ can be joined  by a path in $\Omega(W,\eta)$. Now we prove the existence of a suitable $\delta$ by a contradiction argument, so suppose there are violating sequences $\delta_{i}\rightarrow 0$ and $W_{i}\subset B_{i}$. We choose  (not-necessarily continuous) maps $\Phi_{i}:B_{i}\rightarrow B$ so that $$
\vert d(x,y)-d(\Phi_{i}(x), \Phi_{i}(y))\vert \leq 3\delta_{i}^*, $$
 and the image of $\Phi_{i}$ is $2\delta_{i}^*$-dense. 
Cover $W_{i}$ by a fixed number of $\eta$-balls with centres $x_{i,\mu}$, and consider the points $\Phi_{i}(x_{i,\mu})\in B^{2n}$.  Passing to a subsequence, we can suppose that for each $\mu$ these converge as $i\rightarrow \infty$ to points $z_{\mu}\in B^{2n}$. Let $W_{\infty}\subset B^{2n}$ be the union of the $\eta$-balls with centre $\mu$. 

We want to construct a  map $\phi$ on $\Omega(W_{i},r)$ for large $i$, as in the statement of the Proposition, thus obtaining a contradiction. To simplify the notation we drop the index $i$.  Fix a cover of $B^{2n}\setminus W_{\infty}$ by   $c\eta$-balls centred at points $p_{\nu}\in B^{2n}\setminus W_{\infty}$ and let $1=\sum \chi_{\mu}$ be a smooth partition of unity on $B^{2n}\setminus W_{\infty}$ subordinate to this cover. Choose points $x_{\nu}\in B\setminus W $ such that $d(\Phi(x_{\nu}), p_{\nu})\leq \delta^*$. By hypothesis we have a $(c\eta, \delta)$-chart $\Gamma_{\nu}$ mapping the $c\eta$-ball centred on $x_{\nu}$ to $\bC^{n}$,  with $\Gamma_{\nu}(x_{\nu})=0$. Given an element $g_{\nu}$ of the orthogonal group, we can define $\tilde{\Gamma}_{\nu}(x)= g\Gamma_{\nu}(x)+ p_{\nu}$. Write $\Phi_{\nu}= \Phi\circ \tilde{\Gamma}_{\nu}^{-1}$. By an elementary geometric argument we can choose $g_{\nu}$ so that 
for all $z$ in the domain of $\Phi_{\nu}$ we have  $d(z,\tilde{\Phi}_{\nu}(z))\leq \delta^*$. Now consider a pair of indices $\nu,\nu'$ such that the $c\eta$-balls centred at $p_{\nu}, p_{\nu'}$ intersect. We have an overlap map 
$f_{\nu,\nu'}=\tilde{\Gamma}_{\nu'}\circ \tilde{\Gamma}_{\nu}^{-1}$ defined on a domain which is arbitrarily close to this intersection and we have
$d(f_{\nu,\nu'}(z), z) \leq \delta^*$. By construction this map has the form considered in Lemma \ref{lemmap} and  it follows from that lemma that, after shrinking the domain slightly, $f_{\nu,\nu'}$ is close to the identity in $C^{1}$. Now let $\tilde{\chi}_{\nu}$ be the functions on $B$ defined by $\tilde{\chi}_{\nu}=\chi_{\nu} \circ \tilde{\Gamma}_{\nu}$. We can make $\sum_{\nu} \tilde{\chi}_{\nu}$ as close to $1$ as we please over $B\setminus W$ by making $\delta$ small. Define
$$\phi(z)= \frac{\sum_{\nu} \tilde{\chi}_{\nu}(z) \tilde{\Gamma}_{\nu}(z)}{\sum_{\nu} \tilde{\chi}_{\nu}(z)}. $$
 It is easy to check that for any $\tilde{\theta}$ we can  ensure  $d(x,\phi(x))\leq \tilde{\theta} $, for all $x$, by making $\delta$ small. Fix $\nu$ and consider the composite $f_{\nu}=\phi\circ \tilde{\Gamma}_{\nu}^{-1}$, which is defined on a domain arbitrarily close to the $c\eta$ ball centre  $p_{\nu}$. We can write
$$  f_{\nu}(z)= z + \sum_{\nu'} a_{\nu'}(z) (f_{\nu,\nu'}(z)-z), $$
where the sums runs over indices $\nu'\neq \nu$ such that the corresponding balls intersect and  the weight functions $a_{\nu'}$ satisfy a fixed $C^{1}$ bound. Then Lemma \ref{lemmap} implies that $f_{\nu}$ is $C^{1}$ close to the identity map. Thus to show that $\phi_{*}(J), \phi_{*}(\omega)$ are close to the standard structures, in the sense stated in Proposition \ref{prop2},  it suffices to show the same for $(\tilde\Gamma_{\nu})_{*}(J) , (\tilde\Gamma_{\nu})_{*}(\omega)$, for each $\nu$. Let $J_{\nu}, \omega_{\nu}$ be the structures obtained from the standard structures on $\bC^{n}$ by applying the Euclidean transformation $g_{\nu}$. Then it is immediate that   $(\tilde{\Gamma}_{\nu})_{*}(J) , (\tilde{\Gamma}_{\nu})_{*}(\omega)$ are close in the relevant sense to $J_{\nu}, \omega_{\nu}$. Thus for any $\nu, \nu'$ such that the corresponding balls intersect $J_{\nu}$ must be close to $J_{\nu'}$ and $\omega_{\nu}$ to $\omega_{\nu'}$. The connectivity hypothesis implies that $\Omega(W, r)$  is covered  by balls as above such that any two can be joined by a chain of balls whose successive members intersect.  Since we have a fixed number of balls we see that we can suppose, after perhaps making a single overall change of linear complex structure on $\bC^{n}$ that $(\tilde{\Gamma}_{\nu} )_{*}(J), (\tilde{\Gamma}_{\nu})_{*}(\omega)$ are close to fixed standard structures. This completes the proof of Proposition \ref{prop2}.

\

\begin{prop} \label{prop3} Let $A_{{\rm Euc}}$ be the standard connection on the trivial line bundle $L_{{\rm Euc}}$ over $B^{2n}$ with curvature $\omega_{{\rm Euc}}$. Suppose we add to the hypotheses of  Proposition \ref{prop2} the assumption that $1-I(2B)<\delta$. Then in the conclusion of Proposition \ref{prop2} we can suppose there is a lift $\tilde{\phi}$ of $\phi$ to a bundle map from $L$ to $L_{{\rm Euc}}$ such that $ \vert A_{{\rm Euc}}- \tilde{\phi}_{*}(A)\vert \leq \tilde{\theta}$. 
\end{prop}

The proof involves a number of steps.

{\bf Step 1} We apply Proposition \ref{prop1} with $M=2\tilde{M}$,  $\theta=\tilde{\theta}/2$. This gives us numbers $\eta_{0}, C, \zeta$ and for $\eta\leq \eta_{0}$ a $\psi(\eta)$.

{\bf Step 2} We make some elementary constructions in Euclidean space. Given $C$ as in Step 1 let $S$ be a circle in $\bR^{2n}$ of radius $\lambda \eta$ for $1\leq \lambda\leq C$ and $\eta<< C^{-1}$ with centre at the origin. Choose co-ordinates $(x_{1}, \dots, x_{2n})$ so that $S$ lies in the $(x_{1}, x_{2})$ plane. For $i=3,\dots, 2n$ let $n_{i}$ be the point on the unit sphere with ith. co-ordinate $1$ and all others zero and let $f_{i}(x)= 1- \vert x-n_{i}\vert $. Thus, for small $x$,  $f_{i}$ is an approximation to the $ith.$ co-ordinate function. Put these together to define a map $\underline{f}: B^{2n}\rightarrow \bR^{2n-2}$. 

Next fix a $(2n-2)$ form $\Theta$ on $\bR^{2n-2}$ of integral $1$, supported in the $1/4$ ball and with $\vert \Theta\vert \leq C_{1}$ say. So we have a dilated form $\Theta_{\eta}$ supported in the $\eta/4$-ball, with integral $1$ and $\vert \Theta_{\eta}\vert \leq C_{1} \eta^{2-2n}$. Let $g(t)$ be a function equal to $0$ for $t\leq (\lambda -1/2) \eta$ and to $1$ for $t\geq (\lambda-1/4)\eta$ and define $G:B^{2n}\rightarrow \bR$ by $G(x)= g(\vert x\vert)$. Then it is clear that, if $\eta$ is sufficiently small the $2n-1$ form $dG\wedge \underline{f}^{*}(\Theta_{\eta})$ is supported in the $\eta$-neighbourhood of $S$ and that the fibres of $(\tilde{f}, G): B^{2n} \rightarrow \bR^{2n-1}$ in this neighbourhood are loops which are close to $S$. Let $\alpha$ be $1$ form defined over this $\eta$-neighbourhood of $S$ with $\vert d\alpha\vert\leq n$ and consider the integral
\begin{equation} I= \int \alpha \wedge dG\wedge \underline{f}^{*}(\Theta_{\eta}).\end{equation}
This can be written as a weighted average of the integrals of $\alpha$ over the fibres and it is clear from this that we have
\begin{equation}   \vert I - \int_{S} \alpha \vert \leq C_{2} \eta^{2}, \end{equation}
since any fibres is homologous to $S$ by a surface of area $O(\eta^{2})$.
Let $$C_{3}= 2( C_{2}+ C_{1} C^{2n} \sqrt{2n-2} {\rm Vol} B^{2n}). $$

{\bf Step 3} With $\zeta$ as obtained in Step 1, choose $\eta_{1}$ so that $C_{3} \eta_{1}^{2}\leq \zeta$. Now apply Proposition \ref{prop2} with $\tilde{\eta}=\eta_{1}$.
Choose $\delta<< \eta_{1}$ so that the geometry of the ball $B$ on scale $\eta_{1}$ is very close to that of the Euclidean ball. If we write $W$ for the complement of $\psi(\Omega_{\eta_{1}})$ in $B^{2n}$ we can ensure that $m(\eta_{1}, W)\leq M_{2}$ for some fixed $M_{2}$ depending on $\tilde{M}$.

{\bf Step 4} Now we are in the situation considered in Proposition \ref{prop1}. For clarity we suppress the map $\psi$, so we are regarding $\Omega_{\eta_{1}}$ as a subset $B^{2n}\setminus W$ of $B^{2n}$.  We have a connection $a$ on the line bundle $L\otimes \Lambda^{*}$ over $B^{2n}\setminus W$ and we can arrange that the $L^{p}$ norm of its curvature is less than $\psi(\eta)$ by assuming $\delta$ small.

{\bf Step 5} To complete the proof we need to see that the holonomy of $a$ around any standard $(C,\eta_{1})$-circle $S$ for $W$ is bounded by $\zeta$. First consider the case when the circle has centre at the origin. We use the distance functions from $n_{i}$ to define a Lipschitz map $\underline{f}_{B}: B\rightarrow \bR^{2n}$ with $\vert \nabla \underline{f}_{B}\vert \leq \sqrt{2n-2}$ and the distance from $p$ to define a function $G_{B}: B^{2n} \rightarrow \bR$.
We can suppose (when $\delta$ is small) that the support of $dG_{B}\wedge \underline{f}_{B}^{*}(\Theta_{\eta_{1}})$ is contained in the $\eta_{1}$-neighbourhood of $S$ and hence in $\Omega_{\eta_{1}}$. Then over this support we can suppose that the picture is close as we like to the Euclidean model. (Initially the approximation is in $C^{0}$ but we can introduce a small smoothing.)

By a result of Cheeger and Colding the ball $B$ is homeomorphic to $B^{2n}$ (for small $\delta$) thus the line bundle $L$ is topologically trivial. We choose a trivialisation to give us a connection form $A$. Since the holonomy of the standard connection $A_{{\rm Euc}}$ around $S$ is clearly $O(\eta_{1}^{2})$ it suffices to control the integral of $A$ around $S$.
Now consider 
$$   I_{B}= \int_{B} A\wedge  dG_{B}\wedge \underline{f}_{B}^{*}(\Theta_{\eta_{1}}). $$
By making the approximation close enough we get
$$  \vert I_{B}-\int_{S} A \vert \leq 2C_{2} \eta_{1}^{2}. $$
On the other hand we can integrate by parts to write
$$  I_{B}= \int_{B} G_{B}  dA \wedge \underline{f}_{B}^{*}(\Theta_{\eta_{1}}). $$
This integral is supported on a ball of radius $ C \eta_{1}$ and the integrand is bounded by $\sqrt{2n-2} C_{1}\eta_{1}^{2-2n}$. So we get $\vert I_{B}\vert \leq C^{2n} \sqrt{2n-2} C_{1}\eta_{1}^{2n} \eta_{1}^{2-2n} {\rm Vol}(B^{2n})$ and deduce that the integral of $A$ around $S$ is bounded by $\zeta$ as required.

To explain the main point of this argument. If we knew that the circle $S$ bounds a surface of area $O(\eta_{1}^{2})$ in $B$ then we would get our result directly. While it is perhaps hard to prove the existence of such a surface we reach the same conclusion by this \lq\lq averaging'' argument. 

{\bf Step 6} Using the hypothesis that the twice sized ball is Gromov-Hausdorff close to $2B^{2n}$ we can apply the argument above to any standard loop for $W$.

\subsection{Proof of Proposition \ref{Hormander} and a lower bound on densities}

Recall that our ball $B$ is contained in $X$ and we are supposing that there is a line bundle $L\rightarrow X$ with curvature $\omega$. To prove Proposition \ref{Hormander} we work with a power $L^{k}$ and we write $R=k^{1/2}$. Let $B^{\sharp}$ be the same ball but with the metric scaled by factor $R$, so $B^{\sharp}$ has radius $R$. The first step in the proof is to fix $R$, and to do this we need to review the overall strategy of the argument. We write $\Omega^{\sharp}(W,s)$ for the points of distance {\it in the rescaled metric}, greater than $s$ from $W$ and from the boundary of $B^{\sharp}$. In the argument we will need to choose numbers $r_{0}> r_{1}$ and a cut-off function $\chi$ supported in $\Omega^{\sharp}(W, r_{1})$ and equal to $1$ on $\Omega^{\sharp}(W,r_{0})\cap (1/2) B^{\sharp}$. We will also need to show that there is an \lq\lq approximately holomorphic isometry'' matching up the data over $\Omega^{\sharp}(W, r_{1})$ with the model over $\bC^{n}$. Given this we get $(n+1)$ approximately holomorphic sections
$\sigma_{0}, \dots, \sigma_{n}$ of $L^{k}\rightarrow X$ which we project to genuine holomorphic sections $s_{i}=\sigma_{i}-\tau_{i}$. The ratios
$s_{i}/s_{0}$, for $i=1,\dots, n$ will then give the components of the map $F$. 

The choice of $R$ is coupled to the choice of $r_{0}$. Our cut-off function $\chi$ will have the form $\chi=\chi_{W}\chi_{R}$ where $\chi_{R}$ is a function of compact support on $B^{\sharp}$, equal to $1$ on $(1/2) B^{\sharp}$. Fix  standard functions $g_{R}$ on $\bC^{n}$, supported in the ball of radius $R$ and equal to $1$ on the ball of radius $R/2$. Let
$$ E(R)= 2 \left( \int_{\bC^n} \vert z\vert \vert \nabla g_R\vert^{2} e^{-\vert z\vert^{2}/2}\right)^{1/2}. $$
Then $E(R)$ is a rapidly decreasing function of $R$, say $E(R)= O(e^{-R^{2}/10})$. In our construction, the $L^{2}$ norm of $\tau_{i}$ is controlled by that of $\db \sigma_{i}$. There are various contributions to the  \lq\lq error term'' $\db \sigma_{i}$ but the construction will be such that the $L^{2}$ norm of the contributions from the cut-off function $\chi_{R}$ will be bounded by $E(R)$. 

Given a number $d$, suppose that $m(d/R, W)\leq M$. Then {\it working with the rescaled metric}  the volume  of the $d$-neighbourhood of $W$ is bounded by a fixed multiple of $R^{2n-2} M d^{2}$. Thus if $d= c_{1} (R^{2n-2} M)^{-1/2}$ for a suitable fixed $c_{1}$ this $d$ neighbourhood cannot contain the ball of radius $1/2$, in the rescaled metric, centred at $p$. In other words there is a point $z_{0}$ in this 1/2-ball of distance greater than $d$, in the rescaled metric, from any point of $W$. Similarly, suppose we have a number $r_{0}<1/10$ such that $m(r_{0}/R, W)\leq M$ then if $d_{0}= c_{2} (r_{0}^{2} M R^{2n-2})^{1/2n}$ for a suitable fixed $c_{2}$ then any point $z$ in this 1/2- ball is within distance $d_{0}$ of a point $z'$ in $\Omega^{\sharp}(W, 2 r_{0})$. Given such a a point $z'$,  the $r_{0}$ ball centred at $z'$ lies in the set where the cut-off function is $1$. Thus, writing $\tau_{i}= \sigma_{i}-s_{i}$ we have a fixed estimate for the derivative of $\tau_{i}$ over this ball. Just as in \cite{CDS2}, this will give an estimate
\begin{equation} \vert \tau_{i}(z')\vert \leq c_{3} r_{0}^{-n} \Vert \tau_{i} \Vert_{L^{2,\sharp}}, \end{equation}
for some fixed $c_{3}$. Now for any point $z$ in the $1/2$ ball the derivative estimate for $s_{i}$ gives
\begin{equation}  \vert \ \vert s_{i}(z)\vert - \vert \sigma_{i}(z')\vert \ \vert \leq   c_{3} r_{0}^{-n} \Vert \tau_{i} \Vert_{L^{2,\sharp}}+ c_{4} d_{0} \end{equation}
where $z'$ is a point within distance $d_{0}$ of $z$ and lying in $\Omega^{\sharp}(W, 2r_{0})$. 

The purpose of all this review is now to recall that if $d_{0}$  and $ \vert \tau_{i}(z')\vert $ are small (for all $z$ in the half-ball, with $z'$ chosen as above), then the argument as in \cite{kn:DS}, \cite{CDS2} will give a strictly positive lower bound on $\vert s_{i}\vert$ over the $1/2$-ball. Then we get a derivative bound on $F$ and if further $d_{0}$ and $\vert \tau_{i}(z')\vert$ are sufficiently small compared with $d$
we can show that the map $F$ gives a homeomorphism to its image. (The role of the point $z_{0}$--- \lq\lq far away'' from the set $W$---comes in showing  that the degree of $F$ is $1$.) Likewise for the other statements in Proposition \ref{Hormander}.
 In other words we want
\begin{equation} r_{0}^{1/n} M^{1/2n} R^{(2n-2)/2n} \leq c_{5} (R^{2n-2} M)^{-1/2} \end{equation}
and
\begin{equation} \Vert \db \sigma_{i}\Vert_{L^{2,\sharp}} r_{0}^{-n} \leq c_{6} (R^{2n-2} M)^{-1/2}, \end{equation}
for some fixed, computable $c_{5}, c_{6}$. For the present discussion we just consider the contribution to $\db \sigma_{i}$ from the cut-off $\chi_{R}$, so we replace (6) by
\begin{equation}   E(R)  r_{0}^{-n} \leq (1/10) c_{6} (R^{2n-2} M)^{-1/2}. \end{equation}

Combining (5) and (7), we need
\begin{equation}10^{1/n} E(R)^{1/n} (R^{2n-2} M)^{1/2n} \leq c_{6}^{1/n} c_{5}^{n} M^{-1/2} R^{1-n} (R^{2n-2} M)^{-n/2}.  \end{equation}
The crucial point then is that the rapid decay of $E(R)$ means that this inequality is satisfied for large $R$ (depending on $M$). Thus we fix such an $R$ and then we can fix $r_{0}$ to satisfy (5), (7).

Now that $R$ is fixed we can imagine that in fact $R$ is equal to $1$. This is not truly realistic but purely to simplify notation. What we are saying is that we can now effectively ignore the contribution to $\db \sigma_{i}$ arising from the cut-off function $\chi_{R}$. We also have a fixed $r_{0}$. Examining the argument in Proposition 1 in \cite{CDS2} we see given any $\epsilon>0$  we can choose an $r_{1}<r_{0}$, depending only on $M, r_{0}, \epsilon $, so that if $W\subset B$ is any set with $m(r_{1}, W)\leq M$ then there is a function $\chi_{W}$ on $B$, vanishing on the $r_{1}$ neighbourhood of $W$, equal to $1$ outside the $r_{0}$ neighbourhood and with $\Vert \nabla \chi_{W}\Vert_{L^{2}} \leq \epsilon$. We choose $\epsilon$ so that $$ \epsilon r_{0}^{-n} \leq (1/10) c_{6}  M^{-1/2} $$ where $c_{6}$ is as in (7). This takes care of the contribution to $\db \sigma_{i}$ arising from $\chi_{W}$. Now $r_{1}$ is fixed. It is clear that we can fix a  $\tilde{\theta}$ depending only on $M, r_{1}$ so that if we have an open embedding $\phi:\Omega(W, r_1)\rightarrow \bC^{n}$  with a lift $\tilde{\phi}$ of the kind considered in Propositions \ref{prop2}, \ref{prop3}, with this bound $\tilde{\theta}$, then  by pulling back from the Euclidean model we can construct approximately holomorphic sections, $\sigma'_{i}$, over $\Omega( W, r_{1})$  such that $\Vert \db \sigma'_{i}\Vert \leq \epsilon$,  with $\epsilon$ as above. It is also straightforward to show that we can choose $\tilde{\theta}$ so that if we define $\chi_{R}$ on $\Omega(W, r_1)$ to be the composite $ g\circ \phi$ then the $L^{2}$ norm of
$(\nabla \chi_{R}) \vert \sigma'_{i}\vert$ differs by at most $\epsilon$ from the corresponding term in the Euclidean model (which we bounded by $2E(R)$).
We apply Propositions \ref{prop2}, \ref{prop3}  with $r_{1}$ as chosen above, with this $\tilde{\theta}$ and with $c,M$ as in the statement of Proposition \ref{Hormander}. This gives us a $\delta, \eta$. We take these as the values called for in the statement of Proposition \ref{Hormander}. The proof of that Proposition is now in our hands: we take $\sigma_{i}= \chi_{R} \chi_{W} \sigma'_{i}$.

\begin{cor}\label{cor2}
There are $\kappa(M)>0$ such that if $B$ satisfies the hypotheses of Proposition \ref{Hormander} (with parameters $M,\delta,\eta,c$ as in the statement of that Proposition), and if in addition the centre $p$ lies in $D$ then the volume of $D\cap B$ is at least $\kappa(M) $. \label{cor2}
\end{cor}

The proof is just as in \cite{CDS2}. The image $F(D)$ is a hypersurface in $0.9 \rho B^{2n}$ containing the origin and so must have a definite Euclidean volume. Then the  bound on $\vert \nabla F\vert $ gives a lower bound on the volume of $D\cap B$.

Recall that if $q\in X$ is a point and $r>0$  we write
 \begin{equation}
 V(q,r)= r^{2-2n} {\rm Vol}(D\cap B(q,r)).
 \label{volumeratiowithdivisor}
 \end{equation} 
 We want to build on Corollary \ref{cor2} to get a general lower bound on $V(q,r)$ for points in the the intersection of $D$ and a ball which is Gromov-Hausdorff close to the flat ball. 
 
\begin{prop} \label{prop5}
In the setting above, there are $\delta_{1}, \kappa>0$ such that if $1-I(B)\leq \delta_{1}$ and if $q$ is a point in $D$ with $d(p,q)\leq 1/2$ then
$V(q,r)\geq \kappa$ for all $r$ such that $B(q,r)\subset B(p,1/2)$. 
\end{prop}

For any fixed $q\in D$ the quantity $V(q,r)$ tends to the volume of the unit ball in $\bC^{n-1}$ as $r$ tends to $0$. It follows easily that there is a $q_{0}, r_{0}>0$ such that $V(q_{0}, r_{0})$ is minimised among all $q,r$ such that $B(q,r)\subset B(p, 1/2)$. Let $B^{*}$ be the ball $B(q, r/2)$  scaled to unit size. By a covering argument, just as in Section 2.8 of \cite{CDS2} and using the minimising property, we get a bound on the Minkoswki measure $m(B^{*} \cap D) \leq M_{0}$ where $M_{0}$ is a fixed computable number.   It follows that there is a fixed computable $M$ so that for any $\eta$, if $W_{\eta}$ is the $\eta$-neighbourhood of $B^{*}\cap D$ then $m(\eta, W_{\eta})\leq M$. Let $x\in B^{*}$ be a point which is not in $W_{\eta}$ so the $\eta/2$ ball centred at $x$ does not meet $D$. By the discussion above and Anderson's theorem in \cite{a90} (see also Proposition \ref{prop1.8} in Section 3.2 ) we can choose $\delta_{0}$ so that if $1-I(B)\leq \delta_{0}$ then we have a $(c\eta, \delta')$-chart centred at $x$ for some fixed $c$ and where $\delta'$ can be made as small as we please by making $\delta_{0}$ small. Now take these values of $M,c$ in Proposition \ref{Hormander}. That proposition gives us numbers $\delta, \eta$. We fix $\eta$ in this way and define $W=W_{\eta}$. We choose $\delta_{1}\leq \delta_{0}$ and so small that the hypotheses of Proposition 1 are satisfied. Then Corollary \ref{cor2} gives a lower bound on the volume of $B^{\sharp}\cap D$ which is at most
$2^{2n} V(q,r)$.

\subsection{Volume doubling argument}

We now come to the final phase of technical work in the proof of Theorem \ref{thm3}. So far the cone angle $\beta$ has played no role but we bring it in now in the next Proposition.
In this Proposition we continue to consider a unit ball $B=B(p,1)\subset X$ although in our application we will apply it to smaller balls by scaling.

\begin{prop}\label{goodchart}
Let $C$ be the number given in Proposition \ref{Hormander}. Given $A,\theta>0$ there are $\sigma(C), \gamma(A,\theta), \delta_*(\theta)>0$ with the following effect. Suppose that ${\rm Vol}(B\cap D)\leq A$ and that $1-I(B)\leq \delta_*(\theta)$. Suppose that there is a holomorphic map $F:B\rightarrow \bC^{n}$ with $F(p)=0$  which is a homeomorphism to a domain lying between $ B^{2n}$ and $(1.1) B^{2n} $, satisfying $\vert \nabla F\vert \leq C$ and that $\omega$ is defined by a K\"ahler potential $\phi$ with $\vert \phi\vert \leq C$. If $\beta\geq 1-  \gamma(A,\theta)$ then there is a $(\sigma, \theta)$-chart centred at $p$.
\end{prop}

The crucial thing here is that $\delta,\sigma $ do not depend on $A$. Consider first the case when $D\cap B$ is empty, so we are considering a smooth K\"ahler-Einstein metric.  Then it follows from Anderson's results that there for any fixed $\sigma_{0}<1$ there is a   $ \delta_{{\rm Anderson}}(\theta)$ for which the conclusion holds. To be definite fix $\sigma_{0}=1/2$. We define $$\delta_*(\theta)=(1/2)\delta_{{\rm Anderson}}(\theta/2)\ \ \ 
\sigma(C)= \min (1/2, 1/(2\sqrt{C})). $$ Now suppose that the Proposition is not true, so for fixed $C,A,\theta$ we have a sequence $\beta_{i}\rightarrow
1$ and violating examples $(B_{i},\omega_{i})$. We consider the metric $\omega_{i}$ as defined on the unit ball $B^{2n}$ via the map $F_{i}$. Thus we have $\omega_{i}=\dbd \phi_{i}$ where
$\vert \phi_{i}\vert \leq C$ and $\omega_{i}\geq C^{-1} \omega_{{\rm Euc}}$. Thus we have a fixed bound on the Euclidean volume of the singular set $D_{i}\subset B^{2n}$. This implies that restricting to  a slightly smaller domain $0.9 B^{2n}$ say,  and taking a subsequence, the $D_{i}$ converge to some limiting divisor $D_{\infty}$ and we can choose normalised defining functions $f_{i}$ for $D_{i}$ converging to a non-zero limit $f_{\infty}$.
\begin{lem} \label{lemma4}
Given a divisor $D_{\infty}\subset (0.9)B^{2n}$ there are $\epsilon,K>0$ such that if $g$ is any holomorphic function on $(0.9) B^{2n}$ then
the the supremum of $\vert g\vert$ over $(0.8) B^{2n}$ is bounded by $K$ times the supremum of $\vert g\vert$ outside the $\epsilon$-neighbourhood of $D_{\infty}$.
\end{lem}
This follows from the Cauchy integral formula applied to suitable holomorphic discs.

Now simplify notation by dropping the index $i$ for the moment. The K\"ahler-Einstein equation satisfied by $\phi$ takes the form (without loss of generality we assume here $D$ is in $|-K_X|$)

\begin{equation} \label{local KE equation}    \det \phi_{a\overline{b}}= e^{-\lambda \beta \phi} \vert f \vert^{-2(1-\beta)}\vert U\vert^{2} \end{equation}
where $\lambda \leq 1$, the function $f$ is the chosen local defining function for $D$ and $U$ is some nowhere-vanishing holomorphic function on $(0.9)B^{2n}$.
(The factor $\lambda$ arises from the fact that we are allowing scaling of the metric.) The facts that $\omega\geq C^{-1} \omega_{{\Euc}}$ while the integral of $\omega^{n}$ over $B^{2n}$ is bounded give upper and lower bounds on $\vert U\vert$ outside the $\epsilon$ neighbourhood of $D_{\infty}$, for large $i$. So restricting to a slightly smaller domain $(0.8)B^{2n}$ and applying the Lemma to $U$ and $U^{-1}$ we get upper and lower bounds on $U$. 
Thus we get
\begin{equation} \label{local bound metric}    C^{-1}\omega_{{\rm Euc}}\leq  \omega\leq C' \vert f\vert^{-2(1-\beta)}\omega_{{\rm Euc}}. \end{equation}

Now reinstate the index $i$. The K\"ahler-Einstein equation (\ref{local KE equation}) and the bounds (\ref{local bound metric}) show that (taking a subsequence) the $\phi_{i}$ have a limit $\phi_{\infty}$ outside $D_{\infty}$. Write $\omega_{\infty}=\dbd \phi_{\infty}$. Since $\beta_{i}\rightarrow 1$ we deduce from (\ref{local bound metric}) that
\begin{equation} C^{-1} \omega_{{\rm Euc}}\leq \omega_{\infty} \leq C' \omega_{{\rm Euc}}\end{equation}

The main technical fact we need is:

\begin{prop} \label{prop6}
The limit $\omega_{\infty}$ extends to a  smooth K\"ahler-Einstein metric on $(0.7) B^{2n}$. 
\end{prop} 
 We give a proof in Appendix 2.

Much as in \cite{CDS1}, the bound (\ref{local bound metric}) shows that the for any points $z,w$ in $(0.7)B^{2n}$ the distance $d_{i}(z,w)$  between $z,w$ in the metric $\omega_{i}$ converges uniformly to the distance $d_{\infty}(z,w)$ in the metric $\omega_{\infty}$ as $i\rightarrow \infty$. It also follows easily from (\ref{local bound metric}), and the $C^{\infty}$ convergence away from $D_{\infty}$, that $\omega_{i}$ converges to $\omega_{\infty}$ in $L^{p}$ for any $p$. Since $\sigma\leq 1/(2\sqrt{C})$ the lower bound $\omega_{\infty}\geq C^{-1} \omega_{{\rm Euc}}$ implies that the  $\sigma$ ball centres at the origin defined by the metric $\omega_{\infty}$ is a domain $U$ in $(1/2) B^{2n}$. Since $\sigma\leq 1/2 $ there is a holomorphic homeomorphism
$G$ defined on a neighbourhood of the closure of $U$ with $G(0)=0$ such that
$\vert \vert G(z)-G(w)\vert- d_{\infty}(z,w)\vert \leq \theta/2$ and the $L^{p}$ norm of $G_{*}(\omega_{\infty})-\omega_{{\rm Euc}}$ is less than $\theta/2$. Now the composite $G\circ F_{i}$ gives the desired $(\sigma, \theta)$-chart on $B_{i}$ for large $i$, contradicting our assumption. This completes the proof of Proposition\ref{goodchart}.

\

\begin{prop} \label{A good}
There is a $c>0$ such that for any $\theta, A >0$ we can find a $\delta(\theta)$ and $\gamma(A,\theta)$ such that if $B=B(p,1)$ as above and
\begin{itemize}
\item  $1-I(B)\leq \delta$;
\item  ${\rm Vol}(B\cap D)\leq A$;
\item $\beta>1-\gamma$;
\end{itemize}
then there is a $(c,\theta)$-chart centred at $p$.
\end{prop}

The proof is a matter of putting together what we have done.
Suppose that we have ${\rm Vol}(D\cap B)\leq 2A$ and let $r>0$. Let $Z_{r}\subset B$ be the set of points $x$ where $V(x,r)\geq A$. Then by a standard covering argument there is a fixed and computable $M$, independent of $r,A$,  such that $Z_{r}$ can be covered by $M$ balls of radius $r$. In fact $M=2M_{0}$ where $M_{0}$ is the number considered in the proof of Proposition \ref{prop5}. Proposition \ref{Hormander}
gives us a fixed number $\rho=\rho(M)$. We also take the fixed number $C$ from Proposition \ref{Hormander} and put it into  Proposition \ref{goodchart} to get a number $\sigma(C)$. Set $c=\rho\sigma$. Now feed these numbers $M,c$ into Proposition \ref{Hormander}, to get fixed numbers $\eta(M,c)$ and $\delta(M,c)$.
Take $\delta(\theta)\leq \delta_{1}$ where $\delta_{1}$ is the function in Proposition \ref{prop5}. Also take $\delta(\theta)\leq \delta(M,c)$ and very small compared to the number $\delta_{*}(\theta)$ in Proposition \ref{goodchart}.  Let $\gamma(A,\theta)=\gamma(A,\theta, C)$ with the fixed $C$ above where again $\gamma(A,\theta,C)$ is the function in Proposition \ref{goodchart}. 
 Say that a number $A$ is \lq\lq good'' if the statement of Proposition \ref{A good} is true for $A$, with all other parameters fixed as above.  
It is clear from Corollary \ref{cor2} that  sufficiently small $A$ are good. (For if $A$ is sufficiently small the divisor cannot come within distance $1/2$, say,  of  $p$ and we reduce to the standard situation.) Thus it suffices to show  that if $A$ is good then so is $2A$.

 So suppose that ${\rm Vol}(B\cap D)\leq 2A$ and $A$ is  good. Given $r>0$,  let $Z_{r}$ be the set of points $q\in B$ such that $V(q,r)\geq A$ as above. Given   $\eta$ as above, we let $W$ be the intersection of $Z_{r}$ as $r$ runs from $1$ to $\eta$. Thus, by the covering argument,  $m(\eta, W)\leq M$. 
Consider a point $x\in B$ which is not in $W$. Then for some $r\geq \eta$ we have $V(x,r)\leq A$. By the assumption that $A$ is good and the monotonicity of the $I$ invariant we can apply Proposition \ref{goodchart} to the $r$-ball centred at $x$, after rescaling. If $\delta$ is sufficiently small the hypotheses of Proposition \ref{Hormander} are satisfied so we get a map $F$ on the ball $B(p,\rho)$
and a bound on the K\"ahler potential. This serves as the input to Proposition \ref{goodchart} and we see that $2A$ is good.

\

\subsection{Proofs of Theorems \ref{thm3} and Corollary \ref{cor1}}

We first prove Theorem \ref{thm3}. Fix $\theta$ small, by Proposition \ref{A good} we obtain $\delta=\delta(\theta)$.  
Recall that this is also smaller than the number given by Anderson which controls smooth K\"ahler-Einstein metrics.
Let $p$ be a point in the regular set of the Gromov-Hausdorff limit $Z$. We fix $r$ so that the rescaled $r$-ball $B$ centered at $p$ has $1-I(B)<\delta/2$ . Let $p_{i}\in X_{i}$ be a sequence converging to $p$ and let $B^{\sharp}_{i}$ be the rescaled $r$-ball in $X_{i}$ centre $p_{i}$. Then for large enough $i$ we have $1-I(B^{\sharp}_i)\leq \delta$. The volume of the divisors $D_{i}$ is a fixed number $V$ so after rescaling we have ${\rm Vol} (D_{i}\cap B^{\sharp}_{i})\leq A= V r^{2-2n}$.
Since $\beta_{i}\rightarrow 1$ for large $i$ we know $\beta_i>1-\gamma(A, \theta)$, so we have complex co-ordinates on balls of fixed size centred at $p_{i}$, and the metrics converge in $L^{p}$ by the discussion above. 

Now consider a tangent cone $C(Y)$ to $Z$ at a point $p\in Z$ and let $y\in Y\subset C(Y)$ be a point in the regular set. We fix a radius $r<1$ so that the rescaled $r$ ball $B^{\sharp} $ in $C(Y)$ has $1-I(B^\sharp)\leq \delta/3$. By the definition of the tangent cone we can find a sequence $\lambda_{i}\rightarrow 0$ and points $q_{j}\in Z$ with $d(p,q_{j})=\lambda_{j}$
such the rescaled $r\lambda_{j}$-balls $B^{\sharp}_{j}$ centre  $q_{j}$ converge
to $B^{\sharp}$ in Gromov-Hausdorff distance. Choose $j$ large enough that $1-I(B^{\sharp}_{j})<2\delta/3$. Now with this fixed $j$ choose $i$ large enough that there are rescaled $r\lambda_{j}$ balls $B^{\sharp}_{ij}\subset X_{i}$ with $1-I(B^{\sharp}_{ij})<\delta$. We set $A_{j}=(r\lambda_{j})^{2-2n}V$ and the same argument applies. For large enough $i$ there are holomorphic co-ordinates on a ball of fixed size and the limit as $i\rightarrow \infty$ is a smooth K\"ahler-Einstein metric $\omega_{j}$. By Anderson's result we get complete control and the limit as $j\rightarrow \infty$ is also smooth.

Notice that this argument does not apply to a general  scaled limit $Z'$. Thus for example in the situation above with $B^{\sharp}_{ij}$ if we have some function $i(j)$ we are not able to say anything about the limit 
$$  \lim_{j\rightarrow \infty} B^{\sharp}_{i(j)j}, $$
because $A_{j}$ might tend to infinity. It seems likely that such a limit is smooth but we do not know how to prove that.

To prove Corollary \ref{cor1}, suppose we have a point $p\in Z$ with a tangent cone $\bC_{\alpha} \times\bC^{n-1}$ for $\alpha\neq 0$ and that $p_{i}\in X_{i}$ converge to $p$.  Then, given Theorem \ref{thm3}, the arguments of \cite{CDS2} extend without difficulty to construct holomorphic co-ordinates on balls $B_{i}$ of fixed size $X_{i}$. Let $\Delta$ be a transverse holomorphic disc as in Section 2.6 of \cite{CDS2}. The same argument shows that we have $\alpha \leq 2 \mu_{i} \gamma_{i}$ say, where $\mu$ is the intersection number of $\Delta$ and $D_{i}$. Since $\gamma_{i}\rightarrow 0$ we must have $\mu_{i}\rightarrow \infty$ but then it is easy to show that
the ${\rm Vol}(D_{i}\cap B_{i})$ would tend to infinity which is impossible. The same argument applies to tangent cones. 

\begin{rem}
 Once we have constructed a non-vanishing section of some power
$K^{-m}$ over a fixed size neighbourhood of $p_{i}$ we get a connection $1$-form there. Then we can alternatively argue as in the proof of Proposition \ref{prop3}, taking the wedge product with a $2n-2$ form, rather than restricting to a disc.  
\end{rem}

 \subsection{Proof of Theorem \ref{thm1}}
 Given Theorem \ref{thm3}, we can adapt the results in \cite{kn:DS} to prove Theorem \ref{thm1}. The proof is again based on H\"ormander's technique. There is one thing that requires  clarification. Let $C(Y)$ be a tangent cone of the Gromov-Hausdorff limit $Z$ and $K$ be a compact subset in the regular part of $C(Y)$. In the arguments in \cite{kn:DS} we need to construct a diffeomorphisms $\chi_i$ from a neighborhood of $K$ into $X_i$ so that $|\chi_i^*J_i-J_\infty|$ and $||\chi_i^*\omega_i-\omega_\infty||_{L^p}$ (for some $p>2n$) are small when $i$ is large.  In the setting of \cite{kn:DS} this follows easily from the smooth convergence of metrics. In the setting of Theorem \ref{thm1} this is not immediate since the metrics are not necessarily smooth before we take limit, and we need to patch together the local holomorphic charts obtained from Theorem \ref{thm3}. For this and also for the purpose of directly talking about ``algebraic structures" on a general Gromov-Hausdorff limit (for example, in the general non-algebraic K\"ahler case), we introduce a general notion. Suppose $X_i$ are compact complex n-manifolds with metrics $d_i$, compatible with the complex manifold topology. Suppose that $(X_i,d_i)$ have a Gromov-Hausdorff limit $(Z,d_Z)$, and fix metrics on $X_i \cup Z$ which we also denote by $d_i$. We can define a \emph{presheaf} ${\cal O}_0$ on $Z$ by saying that, for $U\subset X$, a function $f\in {\cal O}_0(U)$ if it is the limit of holomorphic functions over domains in $X_i$, in an obvious sense. Let $\cal O$ be the corresponding sheaf.
 
 \begin{lem}
Suppose that there is some $r > 0$ and functions $\mu_1,\mu_2$ with $\mu_i(t)\rightarrow 0$ as $t\rightarrow 0$, so that for each $q$ in $Z$ and each $i$ there is a holomorphic map $h_{i,q}: B^{2n} \rightarrow X_i$ with the following properties.
 
 $\bullet$  Any point $x$ in $X_i$ with $d_i(x,q) < r$ lies in the image of $h_{i,q}$.
 
 $\bullet$ $d_i(h_{i,q}(z),h_{i,q}(w))\leq \mu_1(|z-w|)$ and $ |z-w|\leq \mu_2(d_i(h_{i,q}(z),h_{i,q}(w))).$
Then $\cal O$ endows $Z$ with the structure of a complex manifold and for sufficiently large $i$ there are diffeomorphisms $f_i : Z \rightarrow X_i$ such that $f_i^*(J_{X_i}) $ converges in $C^\infty$  to $J_Z.$
 \end{lem}
 
The proof is just as in Proposition \ref{prop2}, the point is that the second property controls the $L^\infty$ norm of the holomorphic transition functions, which allows us to take limit and glue together the local charts. Given this lemma one can also deal with convergence of differential forms.

\begin{lem}
Suppose in addition there are differential forms $\theta_i$ on $X_i$ so that for each $q$ some subsequence of $h_{i, q}^*(\theta_i)$ converges in $L^p$. Then there is a subsequence $i'$ such that $f_{i'}^*(\theta_i')$ converges in $L^p$. If $\theta_i=\omega_i$ is a (non-necessarily smooth) positive $(1, 1)$ form which
defines a Riemannian metric on $X_i$ corresponding to $d_i$ and if the subsequence
limits of $h_{i,q}^*(\omega_i)$ are smooth positive $(1,1)$ forms then the limit of $f_i^*(\omega_i)$ is a smooth positive $(1,1)$ form defining a Riemannian metric on $Z$ corresponding to $d_Z$.

\end{lem}

We also have a straightforward extension to the noncompact case. 
\begin{lem} Suppose now that the $X_i$ are not required to be compact and $Z$ is a based limit. Suppose we have a compact set $K\subset Z$ such that for each point $q\in K $ there are maps $h_{i,q}$ as above. Then the analogous statements are true for a neighbourhood of $K$.
 \end{lem}
 
Given Theorem \ref{thm3}, Corollary \ref{cor1} and this Lemma, the proof of Theorem \ref{thm1} goes exactly as in Theorem \ref{thm1} in \cite{kn:DS}.  That the limit metric is smooth on the smooth part of $W$ follows similar arguments as in the proof of Proposition \ref{goodchart} and Proposition \ref{prop6}.

\section{Second proof of Theorem \ref{thm3}}
In this section, we provide a second proof of Theorem \ref{thm3}, using Ricci flow.

\subsection{Almost Einstein/Static flow}
Following \cite{WT12}, we say a sequence of normalized Ricci flows  $(X_i, g_i(t))$
\begin{equation}
 {{\partial g_i(t)}\over {\partial t}} = g_i(t) -Ric(g_i(t)),\qquad t\in [0,1].
 \label{eq:ricciflow1}
\end{equation}  is   \emph{ almost Einstein/static } if  the following holds
\begin{equation}
\lim_{i\rightarrow \infty} \int_{X_i}\; |R(g_i(t))-n| \dVol_{g_i(t)}= 0, \qquad \forall t\in [0,1].\;
\label{eq:ricciflow3}
\end{equation}

Again with $(X_i, D_i)$ as in the introduction, and suppose  $g_i$  is a sequence of  K\"ahler-Einstein metrics on $X_i$ with cone angle $2\pi\beta_i$ along $D_i$ and $\beta_i\rightarrow 1$. 
Suppose we have a Gromov-Hausdorff limit $Z$.  By our Theorem 2 in \cite{CDS1}, there exists a sequence of smooth K\"ahler metrics $g_i'$ in $2\pi c_1(X_i)$ with $Ric(g_i') \geq \beta_i g_i'$ and
\begin{equation}
d_{GH} ((X_i, g_i), (X_i, g_i')) \leq i^{-1}. \label{twosequences}
\end{equation}
Moreover from our proof in \cite{CDS1} we may write the corresponding K\"ahler forms as $\omega_i'=\omega_i+i\p\bp \psi_i$ so that $\psi_i$ converges to zero in $L^\infty(X_i)$ and the $C^k$ norm of $\psi_i$ measured using the metric $\omega_i'$ goes to zero locally away from $D_i$. 
It follows that $(X_i, g_i')$ also converges to $Z$ in the Gromov-Hausdorff topology.\\

It is well-known that the normalized K\"ahler-Ricci flow exits for all time if the initial metric is in the first Chern class of a Fano manifold. For our convenience we denote the normalized K\"ahler-Ricci flow initiated from $g_{i}'$ by $g_i'(t)$, and we denote by $\omega_i'(t)$ the corresponding K\"ahler forms. 

\begin{prop} The sequence of flows $(X_i, g_i'(t))$ is
almost static. \label{thm2.2}
\end{prop}
This is known by \cite{WT12} since the Ricci curvature of $g_i'$ is bigger than $\beta_i\rightarrow 1$.  We include here a brief proof for the  convenience of readers. 
By a direct calculation the scalar curvature evolves as follows (we drop the dependence of $i$
to simplify the presentations):
\[
{{\partial R}\over {\partial t}} = \triangle R + |Ric|^{2} - R = \triangle R + |Ric-g|^{2} + R-n.
\]
Thus
\[
{{\partial (R - n)}\over {\partial t}} \geq  \triangle (R - n) + R-n
\]
By maximum principle,  we have
\[
   R(g_i'(t))-n\geq \min_{X_i}(R(g_i'(0))-n) e^{t}.
\]
Let $\lambda_i(t) = \min_{X_i}R(g_i'(t))$, then $\lambda_i(t)\leq n$ and  by assumption $\lambda_i(0)\geq n\beta_i$.  We have
\begin{equation*}
     n-\lambda_i(t) \leq (n-\lambda_i(0))e^{t}.
\end{equation*}
Note that
\[
\begin{array}{lcl} |R(g_i'(t))-n| & \leq & (R(g_i'(t))- \lambda_i(t)) + | \lambda_i(t) -n|\\
& = &  (R(g_i'(t))-n) + 2 (n- \lambda_i(t))
\\ & \leq & (R(g_i'(t))-n) + 2 (n- \lambda_i(0))e^{t}.
\end{array}
\]
It follows that
\begin{eqnarray*}
&&\int_{X_i}\; |R(g_i'(t))-n|(t) \omega_i'(t)^n \nonumber \\ &\leq& \int_{X_i}  (R(g_i'(t))-n)\omega_i'(t)^n + \int_X 2 (n- \lambda_i(0))e^{t}\omega_i'(t)^n\nonumber \\ &= &2n(1-\beta_i)e^{t}\text{Vol}(X_i, \omega_i).
\end{eqnarray*}
Proposition \ref{thm2.2} follows since $\beta_i\rightarrow 1$.

\subsection{Almost CG convergence and pseudo-locality property}
We first  recall a basic lemma (pseudo-locality) in Ricci flow. 
Suppose $g(x,t)$ is a normalized Ricci flow solution on a manifold $X$  as in Equation (\ref{eq:ricciflow1}). Let $g_0=g(\cdot, 0)$ be the initial metric. 
\begin{prop}[Proposition 3.1 in \cite{WT12}, comparing to Theorem 10.1 in \cite{perelman02}] For every $\alpha \in (0, {1\over {200n}}),\;$ there exist
constants $\underline{\delta}=\underline{\delta}(\alpha)\in(0, \frac{1}{2}), \underline{\epsilon}=\underline{\epsilon} (\alpha)> 0$ with following properties:  suppose that  in $B(p, \underline{\delta}^{-1}, g_0)$ we have
 \begin{equation}
Ric(g_0) \geq 0, \; {\rm and}\; VR(B(p, \underline{\delta}^{-1}, g_0)) \geq 1-\underline{\delta}.
\label{eq:ricciflow2}
\end{equation}
Then for any $x\in B(p, 2, g_0)$ and $t\in (0, \underline{\epsilon}^2]$,  we have
\begin{equation}
 |Rm|(g(x,t)) \leq \alpha t^{-1}+\underline{\epsilon}^{-2},
 \label{eq:ricciflow7}
 \end{equation}
and
\begin{equation}
{\emph {Vol}}_{g(t)} (B(x, \sqrt{t}, g(t)) \geq k' t^{n}, 
\label{eq:ricciflow8}
\end{equation}
where $k'$ is a universal positive constant.
\label{prop3.1}
\end{prop}

\begin{rem}  If we compare this to Theorem 10.1 in \cite{perelman02}, we know that both inequalities (\ref{eq:ricciflow7}) and (\ref{eq:ricciflow8})
also hold for a geometric parabolic box 
\[
   \forall\; t \in (0,\underline{\epsilon}^2],\;\; \forall \; x \in B(p,2, g(t)).
\]
\end{rem}

We now introduce the notion of ``almost Cheeger-Gromov convergence" (compare \cite{CW}) for the convenience of our arguments.
We say a sequence of Riemannian manifolds  $(X_i, p_i, g_i)$ converges in the \emph{almost Cheeger-Gromov sense} to 
a length space $(Y, p, g_\infty)\;$ if the following holds
\begin{enumerate}
\item $Ric(g_{i}) \geq 0;\;$
\item $(X_i,p_i, g_i)$ converges in the Gromov-Hausdorff sense to $(Y, p);\;$
\item There exist  closed  subsets $E_{i}$ in $X_i$  with uniformly bounded (codimension two) Minkowski measure and with limit $E_\infty$, and a smooth metric $g_\infty$ on $Y\setminus E_\infty$,  such that
$(X_i\setminus E_i, g_i)$ converges in the Cheeger-Gromov sense to $(Y \setminus E_{\infty}, g_\infty)$.
 \end{enumerate}

We first state and prove a theorem for general Riemannian manifolds.  We adopt the notation (\ref{isoperi}).

\begin{thm}  There are constants $\delta, \epsilon>0$ with the following effect.  Suppose $(X_i, p_{i}, g_i(\cdot, t))$ is a sequence of almost static flows. If the initial metrics satisfy $I(B(p_i, \delta^{-1}, g_i(\cdot, 0)))\geq 1-\delta$ and $B(p_i, 2,   g_i(\cdot, 0))$ converge in the almost Cheeger-Gromov sense to a limit $B(p, 2, g_\infty(\cdot, 0))$, then $E_{\infty}$ is removable singular set in  $B(p, 1, g_\infty(\cdot,0))$, i.e. $B(p, 1, g_\infty(\cdot,0))$ is smooth and
isometric  to  the limit  of $B(p_i, 1,  g_\infty(\cdot,0))$ endowed with the metric $g_i(\cdot, t)$ for any  $t\in (0, \epsilon^2]$. 
\label{thm3.4}
\end{thm}

We proceed in several steps:\\

\noindent {\bf Step1}:
For simplicity of notation, we denote $B_i(r)$ the ball $B(p_i, r, g_i(\cdot, 0))$. 
By Proposition \ref{prop3.1}, fix $\alpha\in (0, \frac{1}{200n})$, then we have constants $\underline\delta$, $\underline\epsilon$. Then we choose $\delta=2\underline\delta$ and $\epsilon=\underline\epsilon$ so that
\begin{equation}
|Rm|(g_i(x,t)) \leq {\alpha+1\over t},\qquad \forall (x,t) \in B_i(2) \times (0,\underline\epsilon^2].
\label{eq:ricciflow4}
\end{equation}

Now we want to take the limit of the sequence of Ricci flows for positive time. For this we need to clarify that points in $B_i(1)$ has a definite distance away from the boundary of $B_i(2)$ with respect to the metric $g_i(x, t)$ for $t>0$. So we need some control of the variation of the geodesic distance under Ricci flow. 
It is well known that, under the Ricci flow, the distance changes according to integration of Ricci tensor
along the geodesic segment  which realizes the distance.  However, a more careful analysis suggests that such a variation is controlled
by the Ricci curvature at the end of the geodesic segment which realizes the distance.  The following lemma is well-known, which goes back to R. Hamilton \cite{Hamilton} (Section 17) and G. Perelman \cite{perelman02} (Lemma 8.3(b)). 

\begin{lem} 
\label{nonshrinking}
Suppose $(X, g(t))$ is a Ricci flow solution (i.e Equation (\ref{eq:ricciflow1})) on $[0,1]$. Suppose $t'\in [0,1]$, and $q_1, q_2$ are two points in $X$ such that 
\[
   Ric(x, t)\mid_{t=t'} \leq (m-1)C,\qquad {\rm if}\;\; d_{g(t')} (x,q_1) < r  {\rm \ or}\;  d_{g(t')} (x,q_2) < r,
\]
then
\[
{d\over {d\, t}} d_{g(t)}(q_1,q_2) \mid_{t=t'} \geq {1\over 2} d_{g(t')}(q_1,q_2) - (m-1) ({2\over 3}C r + r^{-1}).
\]
\end{lem}

Now we apply this Lemma to our setting, notice by our choice of $\delta$ and Remark 1, for any $q$ in $B_i(2)$,  we have the uniform estimate (\ref{eq:ricciflow4}) for  $t\in (0, \underline\epsilon^2]$ and $x\in B(q, 2, g_i(\cdot, t))$. Set $r= \sqrt{t}$, then 
\[
{d\over {d\, t}} d_{g_i(\cdot, t)}(q_1,q_2) \geq {1\over 2} d_{g_i(\cdot, t')}(q_1,q_2)  - c t^{-{1\over 2}} , \forall q_1, q_2 \in B_i(2), t\in [0,\underline{\epsilon}^2]
\]
for some uniform constant $c.\;$ Integrating this out, we obtain
\begin{equation}
d_{g(t)}(q_1,q_2) \geq  d_{g(0)}(q_1,q_2)  - c t^{{1\over 2}}, \qquad \forall  t \in [0,\underline{\epsilon}^2].
\label{eq:ricciflow9}
\end{equation}
Now we choose $\epsilon_1<\min(\frac{1}{2c}, \underline\epsilon)$, then we know that for all $i$ and any $x\in B_i(1)$, $t\in [0, \epsilon_1]$, we have
\begin{equation} \label{distancetoboundary}
B(x, 1/2, g_i(\cdot, t))\subset B_i(2). 
\end{equation}
By the third condition in the definition of almost Cheeger-Gromov convergece, we may choose a small $d>0$ and choose a point $p_i'\in B_i(1)$ that has distance $d$ away from $D_i$  for all large $i$. Then for any $t_0\in (0, \epsilon_1^2)$, by (\ref{eq:ricciflow4}),  (\ref{distancetoboundary}) and the noncollapsing property, by passing to a subsequence,  we can take a smooth pointed  limit of the sequence of Ricci flows over  $\overline{B_i(1)}\times [t_0, \epsilon_1^2]$ based at $p'_i$.  We denote the limit Ricci flow by 
\[
   {{\partial g_{\infty}(\cdot,t)}\over {\partial t}} = g_{\infty}(\cdot, t) - Ric(g_{\infty}(\cdot, t)), \qquad t \in  [t_{0}, \epsilon_1^2].
\]
Clearly, by letting $t_0$ tend to zero and a diagonal argument,  we can extend this to  a smooth
Ricci flow over $(0, \epsilon_1^2].\;$ \\

By our assumption (\ref{eq:ricciflow3}) on almost static flow, the scalar curvature of the limit metric along the Ricci flow
is constant.  From the evolution equation of scalar curvature,
 we obtain
\[
Ric(g_{\infty}(\cdot,t)) =g_{\infty}( \cdot, t), \forall \ t \in    (0, \epsilon_1^2].
\]
Thus,
\[
   {{\partial g_{\infty}(\cdot, t)}\over {\partial t}} \equiv 0. \qquad \forall \ t  \in (0, \epsilon_1^2].
\]
It follows that the limit metric  at each time slice $t>0$ is the same. From this it also follows that we may take the limit over the larger interval $(0, \underline\epsilon^2]$ and also obtain a constant Einstein limit. We denote this by $(Y_\infty, g_\infty)$. Also we may assume that under the initial metric $g_i(\cdot, 0)$ we have $p_i'$ converges to $p_\infty'$ away from $E_\infty$.   It remains to clarify the relation between $Y_\infty$ and the limit of the initial metrics. \\

\noindent {\bf Step 2}. Now we make use of the assumption on almost Cheeger-Gromov convergence.   For any $q \in B(p, 1, g_{\infty}(\cdot, 0)) \setminus E_{\infty},$  there exist $r_q>0$ and $C_q>0$,  
and  $q_i \in B(p, 1,  g_i(\cdot,0))$
such that $B(q_i, r_q,  g_i(\cdot,0)) $ converges in the Cheeger-Gromov sense to $B(q, r_q, g_\infty(\cdot,0)), $
with
\[
\displaystyle \max_{B_{r_q}(q_i,g_{i}(\cdot,0))} |Rm| (g_{i}(x,0)) \leq  C_{q}.\; 
\]
Following  Theorem 10.3 in Perelman \cite{perelman02} (c.f. \cite{Lu10}, \cite{wang10}), there are constants $t_q>0$ and  $C_q'$ depending on $C_q$ such that 
\[
|Rm|(g_{i}(x,t)) \leq C_q', \qquad \forall x \in  B(q_i, {r_q\over 10}, g_{i}(\cdot,0))\times [0, t_q = {r_q^{2}\over 100} ].
\]
It follows that, in the parabolic box $B(q_i, {r_q \over 10},  g_{i}(\cdot, 0))\times [0,t_q],$
we have
\[
  e^{-2 n C_q' t} g_i(x, 0) \leq g_i(x, t) \leq e^{ 2n C_q' t}  g_i(x, 0),\qquad \forall (x,t) \in  B(q_i, {r_q \over 10},  g_{i}(\cdot,0))\times [0,t_q].
\]
Taking limit as $i \rightarrow \infty,\;$ we obtain a smooth Ricci flow in $B(q, {r_q\over 10}, g_\infty(\cdot, 0))\times [0, t_q]$, and since by Step 1 for $t>0$ the metric $g(\cdot, t)$ is Einstein, so it follows that this limit Ricci flow is trivial in $B(q, {r_q\over 10}, g_\infty(\cdot, 0))\times [0, \underline\epsilon^2]$. By the convergence theory in \cite{cc} we know $B(p, 1, g_\infty(\cdot, 0))\setminus E_\infty$ is connected so we can connect $p_i'$ and $q_i$ by a path in a fixed distance (under the metric $g_i(\cdot, 0)$) away from $E_i$. Then by the previous discussion the distance between $p_i'$ and $q_i$ stay bounded for the metrics $g_i(\cdot, t)$ for all $t\in [0, \underline\epsilon^2]$. From this we obtain an open isometric embedding of Riemannian manifolds $\Phi: (B(p, 1)\setminus E_\infty, g_\infty(\cdot, 0))\rightarrow (Y_\infty, g_\infty)$. \\

\noindent {\bf Step 3}: 
Now we show the image of $\Phi$ is dense. Suppose this is not true, then we can find a ball $B(q, r)$ in $Y_\infty$ that does not intersect the image of $\Phi$. By definition there is some ball $B_i=B(q_i, r, g_i(\cdot, \underline\epsilon^2))$ converging to $B(q, r)$ as $i\rightarrow\infty$. By construction, for any $\sigma>0$, for $i$ sufficiently large  the ball $B_i$ is contained in the $\sigma$-neighborhood $N_\sigma$ of $E_i$ (with respect to the metric $g_i(\cdot, 0)$). By assumption on the Minkowski measure of $E_i$ there is a constant $M>0$ independent of $\sigma $ and $i$ such that 
$$Vol(N_\sigma, g_i(\cdot, 0)))\leq M\sigma^2. $$
 By standard maximum principle we know the scalar curvature of $g_i(\cdot, t)$ is positive, so by the Ricci flow equation the volume of any domain decreases along the Ricci flow. Therefore $$Vol(N_\sigma, g_i(\cdot, \underline\epsilon^2))\leq M\sigma^2. $$
  Taking limit we obtain $Vol(B(q, r))=0$, which is clearly a contradiction. 
   So $\Phi$ has a dense image. Then the metric completion of the image of $\Phi$ using the Riemannian metric is equal to $Y_\infty$. On the other hand, it follows from the general  convergence theory \cite{cc} that $B(p, 1, g_\infty(\cdot, 0))$ is equal to the metric completion of its smooth part using the Riemannian metric. Then we  conclude that $\Phi$ extends  to an isometry between $B(p, 1, g_\infty(\cdot, 0))$ and $Y_\infty$. 
   This finishes the proof of Theorem \ref{thm3.4}. 
   
It follows directly from the above argument that the diameter of $(B_i(1), g_i(\cdot, \underline\epsilon^2))$ is uniformly bounded. Moreover,  by using a rescaling we obtain the following corollary:

\begin{cor}\label{distance nonexpanding} Under the same hypothesis as in Theorem \ref{thm3.4}, there is a constant $D>0$ such that for any $\epsilon>0$ small fixed, and for large enough $i$, any points $x, y$ in $B_i(1)$ with $d_{g_i(\cdot, 0)}(x, y)\leq \epsilon$ have the property that $D^{-1}\epsilon\leq d_{g_i(\cdot, \underline\epsilon^2)}(x, y)\leq D\epsilon$. 
\end{cor}

 We remark one can alternatively prove this via a  covering argument, using Lemma \ref{nonshrinking} and the deceasing of volume along the Ricci flow. While we do not need a sharp estimate here, a more elaborate strategy on obtaining a sharp upper bound can be found in \cite{WT12}.\\

In our applications below we need a variant of Theorem \ref{thm3.4}.

\begin{prop} \label{almost Euclidean flow}
In Theorem \ref{thm3.4}, if we replace the assumption ``almost static" by that  for any fixed $R\in (0, \infty)$,  $I(B_i(R))$ tends to zero as $i$ tends to infinity. Then the same conclusion holds for $B_i(1)$.
\end{prop}
The essential point that the ``almost static" condition is used in the above proof is to show the limit Ricci flow is a constant flow on $(0, \underline\epsilon^2]$. Fix  $\alpha\in (0, \frac{1}{200n})$ we obtain the constants $\underline{\delta}$ and $\underline{\epsilon}$ as in Proposition \ref{prop3.1}. For any $R>0$, by assumption we know that for $i$ large enough, 
\[
VR(B_i(R)) \geq 1 -\underline{\delta}.
\]
We define a rescaled flow  
$$g_i'(x, t)=R^{-2}\underline\delta^{-2}g_i(x, R^2\underline\delta^2 t), $$
then we apply Proposition \ref{prop3.1} to $B(p_i, \underline\delta^{-1},g_i'(\cdot, 0))$ to obtain
\begin{equation*}
 |Rm|(g_{i}'(x,t)) \leq \alpha t^{-1}+\underline{\epsilon}^{-2}, \qquad \forall \;(x, t) \in B(p_i, 2, g_i'(\cdot, 0))\times (0,\underline{\epsilon}^2]. 
 \label{eq:ricciflow5}
\end{equation*}
 Rescale back we get
\begin{equation*}
 |Rm|(g_{i}(x,t)) \leq \alpha t^{-1}+R^{-2}\underline\delta^{-2}\underline{\epsilon}^{-2}, \qquad \forall \;(x, t) \in B_i(2R\underline\delta)\times (0,R^2\underline\delta^2\underline{\epsilon}^2]
\end{equation*}
It follows that the limit Ricci flow $g_\infty(\cdot, t)$: 
\begin{equation*} {{\partial g_\infty(\cdot, t)}\over {\partial t}} = - Ric( g_\infty(\cdot, t)), t\in (0, \infty)
\label{eq:ricciflow6}
\end{equation*}
satisfies that 
\begin{equation*}
 |Rm|(g_{\infty}(x,t)) \leq \alpha t^{-1}+R^{-2}\underline\delta^{-2}\underline{\epsilon}^{-2}
 \end{equation*}
 for all $(x, t)$. 
 For a fixed $(x, t)$, we first fix $\alpha$ and  let $R\rightarrow\infty$ so
 \begin{equation*}
 |Rm|(g_{\infty}(x,t)) \leq \alpha t^{-1}. 
 \end{equation*} 
 Then let $\alpha\rightarrow 0$ we obtain that
  \begin{equation*}
 |Rm|(g_{\infty}(x,t))=0. 
 \end{equation*} 
 It then follows that the limit Ricci flow is constant flat on $(0, \infty)$. Then the rest of  proof of Proposition \ref{almost Euclidean flow} goes exactly as in Theorem \ref{thm3.4}.\\

 Now we apply the above arguments in the K\"ahler case, which is our main interest in this paper.  So now we suppose $(X_i,  D_i, g_i)$ is a  sequence of K\"ahler-Einstein manifolds with cone angle $2\pi \beta_i$ along $D_i\in |-\lambda K_{X_i}|$ for a fixed $\lambda$.  We  do not assume $g_i$ is normalized but we only allow finite rescaling so that we have $Vol(X_i) \leq V$ for a constant $V>0$. Let $(Z, p)$ be the Gromov-Hausdorff limit. Suppose we have a bound $m(B(p_i, 1, g_i)\cap D_i)\leq M$. Let $D_\infty$ be  the limit of $D_i\cap \overline{B(p_i, 1, g_i)}$.  Then $D_\infty$ is closed and 
 $m(D_{\infty})\leq M$.  

\begin{prop}  \label{prop1.8}
 There are uniform constants $K$ and $\delta_1$ such that if $I(B(p, 2))>1-\delta_1$, then for all large $i$ and $ x \in B(p_i, 2,  g_i)\setminus D_i$, we have
\[
   |Rm|(g_{i}(x))  \leq {K\over {d_{g_i}(x, D_i)^2}}. 
\]
In particular, the metric on $B(p, 1)\setminus D_\infty$ is a smooth K\"ahler-Einstein metric $g_\infty$, and the convergence is almost Cheeger-Gromov.
\end{prop}

This is certainly well-known, and follows from the gap theorem of Anderson \cite{a90}. For the convenience we outline a proof here. For any $x\in B(p, 2)\setminus D_\infty$, we denote $r_x= d(x, D_\infty)>0. $
Then $
B(x, r_x) \subset B(p,4)\setminus D_\infty.
$
By assumption, $
VR(B(x, r_x)) \geq 1 -\delta_1.
$
It then follows from \cite{a90}  that if $\delta_1$ is smaller than a universal constant depending only on the dimension, then there is a constant $K>0$ such that for $i$ sufficiently large we have
\[
 |Rm(g_i(x'))|  \leq {K\over r_x^2}, \qquad \forall\ x' \in B(x, {r_x \over 2}).
\]
Then the proposition follows easily.\\

Let $\delta_2=\frac{1}{2}\min(\delta, \delta_1)$, where $\delta$ is the constant in Theorem \ref{thm3.4} and $\delta_1$ is the constant in Proposition \ref{prop1.8}. Then 

 \begin{prop} \label{Ricciflow Kahler} Suppose $g_i$ satisfies  $I(B(p_i, \delta_2^{-1}, g_i))\geq 1-\delta_2$ and suppose there are constants $V>0, M>0$ such that $Vol(X_i)\leq V$ and  $m(B(p_i, 2, g_i)\cap D_i)\leq M$ for all $i$, then $B(p, 1)$ is smooth, and  there are constants $r>0$ and $\kappa>0$ with the following effect. For large $i$ we can find a holomorphic chart on $B(p_i, r, g_i)$ so that the K\"ahler form $\omega_i$ satisfies $\omega_i\geq \kappa\cdot \omega_{Euc}$, and  as $i$ tends to infinity $\omega_i$ converges in $L^p$ for any $p>1$. 
 \end{prop}
 
 A remark is in order:
\begin{rem} In \cite{WT12} it is proved that for a sequence of almost static Ricci flows, the Gromov-Hausdorff limit of the initial metrics $(X_i, g_i')$ has the property that the regular part is open and there is no codimension two singularity.   In view of the approximation
result  (\ref{twosequences}), we know that the Gromov-Hausdorff limit of $(X_i, g_i)$ also has the same property. However, for our
purpose of running the H\"ormander argument to prove Theorem 1, this is not enough since we need a stronger convergence (Theorem \ref{thm3}) of the original K\"ahler-Einstein metrics $g_i$ with cone angles at a regular point in the limit.
\end{rem}

 The rest of this subsection is devoted to the proof of this Proposition.    As before we may choose smooth approximations $g_i'$ with positive Ricci curvature and with the same Gromov-Hausdorff limit $Z$. Moreover, the K\"ahler forms satisfy $\omega_i'=\omega_i+i\p\bp\psi_i$, with $\psi_i$ converging to zero uniformly on $X_i$, and smoothly on a fixed distance away from $D_i$ (with respect to the metric defined by either $\omega_i$ or $\omega_i'$).  Denote by $g_i(\cdot, t)$ the Ricci flow solution to Equation (\ref{eq:ricciflow3}), with initial metric $g_i'$.  Then it follows from Proposition \ref{thm2.2} that this sequence of Ricci flow is almost static, and it is also clear that we can ensure the assumptions of the Proposition is satisfied also by $g_i'$ (with the constant $\delta_2$ replaced by $2\delta_2$).  By Theorem \ref{thm3.4} we obtain the smoothness of $B(p, 1)$, and it suffices to establish the statement about the $L^p$ convergence of K\"ahler forms.
 
 By Theorem \ref{thm3.4}  we know  $B(p_i, 1, g_i(\cdot, 0))$ endowed with K\"ahler metrics $g_i(\cdot,\underline\epsilon^2)$ converges smoothly to $Y_\infty$. It is clear we may assume the complex structures also converge. So we can find a holomorphic chart on a definite size ball around $p_i$ with respect to the metric  $g_i(\cdot, \underline\epsilon^2))$ so that the K\"ahler metrics $\omega_i(\cdot, \underline\epsilon^2)$  converges smoothly a limit metric $\omega$. Without loss of generality we may assume the holomorphic chart is the standard unit ball in $\bC^n$ and we denote it by $B$. Let $D_\infty$ be the limit of $D_i$'s. 
Since
$$Vol(D_i\cap B, \omega_i(\cdot, \underline\epsilon^2))\leq Vol(D_i, \omega_i(\cdot, \underline\epsilon^2))=Vol(D_i, \omega_i)\leq n\lambda V,$$
 we see that $D_\infty$ is a divisor with finite multiplicity.  It follows  from  Corollary \ref{distance nonexpanding} that on a fixed scale the metrics $\omega_i(\cdot, 0)$ is uniformly equivalent to $\omega$ for $i$ sufficiently large. So by the almost Cheeger-Gromov convergence we may assume for $t\in [0, \underline\epsilon^2]$, $\omega_i(t)$ converges smoothly locally in a fixed distance away from $D_\infty$ (with respect to the metric $\omega$). 
In $B$ we write
\[
\omega=i\partial \bar \partial \varphi_\infty, 
\]
and
\[
\omega_i(\cdot, t) =  i \partial \bar \partial \varphi_i(\cdot, t). 
\]
 For $t=\underline{\epsilon}^2$ we may assume $\varphi_i(\underline\epsilon^2)$ converges smoothly to $\varphi_\infty$. By our choice of $g_i'$ we know 
\[
\omega_i=i\partial \bar\partial \varphi_i, 
\]
with $\varphi_i=\varphi_i(\cdot, 0)-\psi_i$ and $\lim_{i\rightarrow\infty} |\psi_i|_{L^\infty}=0$. 

We first prove the following Lemma: 

\begin{lem} \label{C0bound}  There are  constants $r_1>0$, $C_1$ such that we can choose the local potentials $\varphi_i$ such that on $r_1B$, 
\[
|\varphi_i|\leq C_1.
\]
\end{lem}
From the Ricci flow equation $\frac{\partial \omega_i(\cdot, t)}{\partial t}=Ric(\omega_i(\cdot, t))-\omega_i(\cdot, t)$ we see that locally
\[
\frac{\partial \varphi_i(\cdot, t)}{\partial t}=\log \frac{\omega_i(\cdot, t)^n}{\omega^n}+\varphi_i(\cdot, t)+f_i(\cdot, t),
\]
where $f_i(\cdot, t)$ is certain pluriharmonic function on $B$. Rewrite this by
\[
\frac{\partial}{\partial t}( e^{-t}\varphi_i(\cdot, t))=e^{-t}(\log \frac{\omega_i(\cdot, t)^n}{\omega^n}+f_i(\cdot, t)).
\]
Since $\varphi_i(\cdot, \underline{\epsilon}^2)$ converges smoothly to $\varphi_\infty$,  and for any $t\in [0, 1]$, $\omega_i(t)$ converges smoothly to $\omega$ locally away from $D_\infty$, it is then easy to see that by changing $\varphi_i(\cdot, 0)$ by a pluriharmonic function on $B$ we may assume $\varphi_i(\cdot, 0)$ converges smoothly to $\varphi_\infty$ locally away from $D_\infty$. Moreover by the positive of scalar curvature  the volume form $\omega_i(\cdot, t)/\omega^n$ is uniformly bounded from below on $B$. So $\varphi_i(\cdot, 0)$ is uniformly bounded from above on $B$. 
  By our choice of $g_i'$ it follows that $\varphi_i$ also converges smoothly to $\varphi_\infty$ locally away from $D_\infty$, and is uniformly bounded above by $C_0$ in $B$.  

Now we recall the standard H\"ormander Lemma.

\begin{lem} [\cite{Hormander}] \label{lem3.6} There are constants $s \in (0, 2^{-1})$, $A>0$ depending only on $n$, such that for every function $\phi$ on $B$ with $i\p\bp\phi>0$, $\sup_B\phi \leq 1$, and $\phi(0)=0$, we have 
\[
\int_{sB} \; e^{-2\phi} \omega^n\leq A .
\]
\end{lem}

By applying the H\"ormander Lemma to a point very close to $0$, we may assume there are $r_2>0$, $C_2>0$ so that for all $i$, 
\begin{equation} \label{hormanderlemma}
\int_{r_2B} e^{-2\varphi_i}\omega^n\leq C_2. 
\end{equation}

Let $F_i$ be the normalized defining equation of $D_i$ which converges to a nonzero limit $F_\infty$ defining $D_\infty$. Then  the K\"ahler-Einstein equation takes the form (without loss of generality we assume that $D_i \in |-K_{X_i}|$): 
\begin{equation}
 \left( i\partial \bar \partial \varphi_i\right)^n = e^{-\gamma_i\beta_i \varphi_i+ H_i-h_\omega} {1\over |F_i|^{2-2\beta_i}}   \omega^n
 \label{eq:conicKE1}
\end{equation}
where $H_i(z)$ is a pluriharmonic function on $B$, and $\gamma_i\leq 1$ is due to possible rescaling.  
Since we have chosen $\varphi_i$ so that it converges smoothly to $\varphi_\infty$ away from $D_\infty$. It follows that $H_i$ is uniformly bounded on $B$ (similar to the proof of Lemma \ref{lemma4}).  For the term $F_i$ since it is normalized to converge,  and since $\beta_i\rightarrow1$,  for any $p$ there is a $C(p)$ such that for $i$ large enough
\[
    \int_{B}\; {1\over {|F_i|^{2(1-\beta_i)p}}} \omega^n \leq C(p).
\]
Together with (\ref{hormanderlemma}) this implies there is a constant $C_3>0$ such that  for $i$ large enough
\[
\int_{r_2B} \; e^{-\frac{1}{2}\varphi_i}\omega_{\varphi_i}^n \leq C_3.
\]
Since $\varphi_i$ is uniformly bounded above  this implies 
\[
\int_{rB} \;  \varphi_i^2 \omega_{\varphi_i}^n \leq (C_0^2+8C_3)Vol(B).
\]
Since \[
\triangle_{\omega_{\varphi_i}}\; (-\varphi_i) \geq -n, 
\]
and  $\omega_{\varphi_i}$ has uniformly bounded Sobolev constant, a local Moser iteration implies that there is a constant $C_4>0$ so that in $\frac{r_2}{2}B$ for sufficiently large $i$, 
\[
\varphi_i \geq -C_4.
\]
Then we can choose $r_1=r_2/2$ and $C_1=C_4$, then Lemma \ref{C0bound} follows.\\

From Lemma \ref{C0bound}, by a local Moser iteration using the Bochner formula, as in the proof of Proposition 24 in \cite{CDS2}, we obtain
that there
exists a constant $\kappa$ such that on $\frac{r_1}{2}B$
$$ \omega_{i}\geq \kappa \omega.
$$

In the proof of Lemma \ref{C0bound}  we have obtained  that for any $p>1$, and $i$ sufficiently large, $\omega_i^n$ is uniformly bounded in $L^p$. Then the above lower bound implies that $\omega_i$ has a uniform $L^p$ bound. On the other hand, $\omega_i$ converges smoothly to $\omega_\infty$ locally away from $D_\infty$. Now by Corollary \ref{distance nonexpanding} we may choose a $r>0$ so that the ball $B(p_i, r, g_i(\cdot,0))$ is contained in $\frac{r_1}{2}B$, then  Proposition \ref{Ricciflow Kahler} follows.

By the well-known monotonicity of volume of a subvariety in $\bC^n$ (c.f. Proposition 14 in \cite{CDS2}) we obtain the following corollary

\begin{cor} There is a constant $c_0>0$ such that for any $q\in D_i$ with $B(q,s,g_i) \subset \frac{r}{4}B$, we have
\[
 V(q, s, g_i)>c_0. 
\]
\label{prop3.8}
\end{cor}

\subsection{Proof of Theorem \ref{thm3}}

As in Section 2, one of the key ingredients that we need is a lower bound on the volume density of the divisor when the $I$ invariant is close to one.  The following is a restatement of Proposition \ref{prop5}. The notations are as in (\ref{isoperi}), (\ref{volumeratiowithdivisor}). 
\begin{lem} \label{lem1.5} There are small constants $\delta_3>0$ and $c_1>0$ with the following effect. Suppose a ball $B(p, 2)$ in  a K\"ahler-Einstein manifold  $(X, g)$ with cone singularities along a divisor $D$  has $I(B(p, 2)) > 1-\delta_3$. Then  for any $q\in B(p, 1)\cap D$ with $B(q, r)\subset B(p, 2)$ we have 
\[
V(q,r,g) \geq c_1.
\]
\end{lem}

We prove this by contradiction. 
If the lemma is false, the we can find a sequence of K\"ahler-Einstein manifolds $(X_i, g_i)$ with cone singularities along $D_i$, points  $q_i\in X_i$ such that $I(B(q_i, 4, g_i))$ tends to $1$, and we can find $p_i\in B(q_i, 2, g_i)\cap D$, $s_i>0$ so that $B(p_i, s_i, g_i)\subset B(q_i, 2,  g_i)$, and  
\[
 \lim_{i\rightarrow\infty} V(p_i, s_i, g_i)= 0.
  \]
Following  a point selection argument as in the proof of Proposition \ref{prop5}, we may assume that  
\begin{enumerate}
\item $ I(B(p_i, 2s_i, g_{i}) )\rightarrow 1$;
\item  $k_i=V(p_i,2s_i, g_{i})\rightarrow 0$;
\item For any $\mu< 1$ and $y \in B(p_i, s_i,  g_{i})$, we have $V(y, \mu  s_i, g_{i}) \geq  k_{i}.\;$
\end{enumerate}
Let $(Y, p_\infty)$ be  the Gromov-Hausdorff limit of
$(X_i, p_i, s_i^{-1} g_i)$.  Then by a covering argument as in Section 2.8 of \cite{CDS2} or in the proof of Proposition \ref{prop5}, the set $D_i\cap B(p_i, 2, s_i^{-1}g_i)$ has uniformly bounded Minkowski measure. 

Now there are two cases.   If $s_i$ does not tend to zero,, so that we can apply Proposition \ref{Ricciflow Kahler} and Corollary \ref{prop3.8} to  obtain a lower bound for the volume density directly $k_i\geq c_0>0$, which causes a contradiction. 

Now we assume $s_i$ tends to zero. By definition we know $VR(B(p_\infty, R))=1$ for all $R>0$, it follows that $Y$ is isometric to the Euclidean space.
Then we can as before approximate $g_i$ by smooth metrics $g_i'$ so that the pointed limit $B(p_i, 2, s_i^{-1}g_i')$ is also $Y$. Then we can run Ricci flow from $s_i'g_i'$ and apply Proposition \ref{almost Euclidean flow}.  Then we want to follow the arguments in the proof of Proposition \ref{Ricciflow Kahler}.  Note that we can not directly apply Corollary \ref{prop3.8} to obtain a contradiction, since a priori it is not clear that the volume of $B(p_i, 2, s_i^{-1}g_i)\cap D_i$ with respect to the flowed metric $\omega_i(\cdot, \underline\epsilon^2)$ is still uniformly bounded so we do not know the limit $D_\infty$ has finite multiplicity. However, we can apply Corollary \ref{distance nonexpanding}, and using our choice of $p_i$'s to conclude that $D_\infty$ must have have zero Minkowski measure. This is clearly a constradiction if $D_i$ has bounded volume with respect to the metric $\omega_i(\cdot, \underline\epsilon^2)$, since then $D_\infty$ must be a divisor with nonzero Minkowski measure. On the other hand if $D_i$ does not have bounded volume, then we can always take a piece $D_i'$ of $D_i$ with 
bounded volume and take a limit $D_\infty'$ with nonzero Minkowski measure, but $D_\infty'$ is clearly contained in $D_\infty$, so we again arrive at a contradiction. 

Now we prove Theorem \ref{thm3}.  We first treat a non-rescaled limit $Z$. For any $p\in \mathcal R$, by a fixed scaling we may assume without loss of generality that $1-I(B(p, \delta_2^{-1}))<\min(\delta_2, \delta_3)$. Then for $i$ large enough, we could find $p_i\in X_i$ converging to $p$ and such that $1-I(B(p_i, \delta_2^{-1}, g_i))<\min(\delta_2, \delta_3)$. Then since the total volume of $D_i$ is uniformly bounded, we can apply  Lemma \ref{lem1.5} to obtain a uniform bound on the Minkowski measure of $B(p_i, 2, g_i)\cap D_i$, then we can apply Proposition \ref{Ricciflow Kahler} to conclude that  $B(p, 2)$ is open and smooth, and the existence of a local holomorphic coordinate chart so that the K\"ahler forms converge in $L^p$ for any $p>1$.

Now suppose $Y$ is an iterated tangent cone of  $Z$. Then for any regular point $q\in Y$, we may choose $q_i\in Z$ and $\lambda_i\rightarrow 0$, such that for any $r>0$, $B(q_j, r, \lambda_j^{-2}g_\infty)$ converges in the Gromov-Hausdorff sense to $B(q, r)$. Again by consideration of volume ratio it follows that for $r>0$ sufficiently small and for $j$ large enough we have $ 1-I(B(q_j, s\delta_2^{-1}\lambda_j, g_\infty))<\min(\delta_2, \delta_3)$. By definition for fixed $j$ we may choose $q_{ij}\in X_i$ so that as $i$ tends to infinity $B(q_{ij}, s \delta_2^{-1}\lambda_j, g_i)$ converges to $B(q_j, s \delta_2^{-1}\lambda_j, g_\infty)$. Then we may apply the discussion in the non-rescaled case to conclude that $B(q_j, s\lambda_j, g_\infty)$ is smooth, and the convergence is in $L^p$ for any $p$. Now since $q$ is regular, it follows from the gap theorem of Anderson again that the convergence from $B(q_j, s, \lambda_j^{-1}g_\infty)$ to $B(q, s)$ is smooth. From this one easily sees that Theorem \ref{thm3} also holds for $Y$. \\

\section{Automorphism group and Futaki invariant}
In this section we prove two results about algebro-geometric properties of the Gromov-Hausdorff limit of K\"ahler-Einstein metrics with cone singularities, one about the automorphism group, and the other about the Futaki invariant.  From \cite{CDS2} and  the previous sections in this paper we have established that in our setting the Gromov-Hausdorff limit is a $\bQ$-Fano variety $W$, together with a Weil divisor $\Delta$ and a number--``cone angle" $\beta\in (0, 1]$. For the convenience of readers we start by recalling some definitions in \cite{CDS2}. Let $W$ be an $n$ dimensional $\bQ$-Fano variety, embedded into $\bC\bP^N$ by the line bundle $K_W^{-m}$ for some $m$. A \emph{smooth K\"ahler metric} on $W$ in the class $2\pi c_1(W)$ is a K\"ahler metric $\omega$ on the smooth part $W_0$ of $W$ that is of the form $m^{-1}\omega_{FS}+i\p\bp \phi$,  where $\omega_{FS}$ is the restriction of the Fubini-Study metric and $\phi$ is a continuous function on $W$, smooth on $W_0$.   Such $\omega$ also defines a metric on $K_W^{-1}$ (unique up to a constant multiple) which is continuous on $W$ and smooth on $W_0$. We often write $\omega=\omega_h$. A metric $h$ on $K_W^{-1}$ defines a volume form $\Omega_h$ on $W_0$, and the K\"ahler metric is \emph{weak K\"ahler-Einstein} if the equation
 \begin{equation} \label{weak KE}
 \omega_h^n=\Omega_h
 \end{equation}
 holds on $W_0$.

Fix $\lambda\geq 0$. As in \cite{CDS2} let $\Delta$ be a Weil divisor in $W$ so that the intersection $\Delta^{(0)}=\Delta \cap W_0$ is defined by a holomorphic section $s$ of $K_W^{-\lambda}$ (when $\lambda=0$, $\Delta$ is empty).  Fix $\beta\in(0, 1]$ and we assume the pair $(W, (1-\beta)\Delta)$ is KLT (Kawamata log terminal) as discussed in \cite{CDS2}. A \emph{weak conical K\"ahler-Einstein metric} for the triple $(W, \Delta, \beta)$ is a continuous metric $h$ on $K_W^{-1}$, which is smooth on $W_0\setminus \text{supp} \Delta^{(0)}$, and  satisfies the equation
\begin{equation} \label{weak conic KE}
\omega_h^n=\Omega_h |s|_h^{2(\beta-1)} 
\end{equation}
on $W_0\setminus \text{supp} \Delta^{(0)}$.

The above definition exactly suits our purpose, but  we remark that there are  weaker notions of K\"ahler metrics on singular varieties. For K\"ahler-Einstein metrics these are all equivalent, see for example the recent work \cite{BBEGZ} for more details.

We now state the first result of this section

\begin{thm} \label{reductive}
If $(W, \Delta, \beta)$ admits a weak conical K\"ahler-Einstein metric, then 
$\Aut(W, \Delta)$ is reductive. 
\end{thm}
In the case $\lambda=0$ by our convention $\Delta$ is empty, so the statement becomes simply that $\Aut(W)$ is reductive. We will prove the case when $\lambda>0$, and the case $\lambda=0$ is similar.

Theorem \ref{reductive} reduces to  the classical Matsushima theorem when $W$ is smooth and $\Delta$ is empty. The original proof runs on the level of Lie algebra, which relies on an integration by part argument. This seems difficult to extend to our general setting, as we do not have much information on the metric around the singularities. Instead we make use of recent results on complex analysis developed  by Berndtsson \cite{Bern}  and others. The main point is that instead of working on its Lie algebra we directly study the group $\Aut(W, \Delta)$. \\

 We begin with some preparations in pluripotential theory.  Recall an upper semi-continuous function on a complex manifold is \emph{plurisubharmonic} if it restricts to a subharmonic function on any holomophic disk.  Let $(W, \Delta)$ be as before and fix  a K\"ahler metric $\omega$ on $W$.  Let $\H_\infty$ be the space of bounded $\omega$-plurisubharmonic functions on $W$, i.e. bounded upper semi-continuous functions $\phi$ on $W_0$ such that locally near each point there is an open neighborhood $U$ on which one can write $\omega=i\p\bp \phi_0$ and $\phi_0+\phi$ is plurisubharmonic on $U$. When $W$ is smooth, $\H_\infty$ contains the usual space $\H$ of smooth K\"ahler potentials in $[\omega]$. 
 As usual any $\phi\in \H_\infty$ defines a bounded metric $h_\phi$ on $K_W^{-1}$,  and a current $\omega_\phi=\omega+i \p\bp \phi$ which we call a \emph{weak K\"ahler metric}. It follows from the definition that the potential of a weak conical K\"ahler-Einstein metric lies in $\H_\infty$.

 Now the point is that Equation (\ref{weak conic KE}) is also well-defined for  weak K\"ahler metrics  if we regard both sides as measures on $W_0$. To see this, we notice that there is no problem with the right hand side as $\phi$  defines a bounded measure $\Omega_{h_\phi}$ on $W_0$. For the left hand side we need to make sense of wedge product of currents. In general one can not do this, but by Bedford-Taylor \cite{BT} the wedge product of the form $\omega^i\wedge \omega_\phi^j$ is well defined as a closed positive current on $W_0$. In particular  any $\phi \in \H_\infty$ defines a Monge-Amp\`ere measure $(\omega+i\p\bp\phi)^n$ on $W_0$.   Thus we can understand  Equation (\ref{weak conic KE}) as an equation for measures on $W_0$. By pushing forward through the inclusion $W_0\hookrightarrow W$ we can also understand these as measures on $W$.
 
  An alternative way to think of the space $\H_\infty$ is to consider the log resolution of singularities $p: W'\rightarrow (W, (1-\beta)\Delta)$, realized as a sequence of embedded blow-ups of $(W, \Delta)\hookrightarrow \bC\bP^N$. Then the restricted K\"ahler metric $m^{-1}\omega_{FS}$ pulls back to a smooth $(1,1)$ form $\omega'$ on $W'$, positive definite exactly away from the exceptional divisors.  Then by the general extension theorem of plurisubharmonic functions $\H_\infty$ can be identified with $\H'_\infty$, the space of $\omega'$-plurisubharmonic functions on $W'$. Then we can make sense of Equation (\ref{weak conic KE})  as measures on $W'$. The KLT condition ensures that the right hand side pulls back to an $L^p$  volume form on $W'$ for some $p>1$.  
  
  One technical tool in the study of  weak (conical) K\"ahler-Einstein metrics is the convexity of a functional defined by Ding \cite{Ding}, as discovered by Berndtsson \cite{Bern}. This should be compared with the more usual Mabuchi functional defined on any K\"ahler class for which the convexity follows formally from an infinite dimensional moment map picture. The advantage of Ding functional is that it requires less regularity on the potential function so is more amenable to the general setting. Given $W, \Delta, \beta$ as before, the Ding functional is defined by
\begin{equation} \label{definition Ding functional}
\Di(\phi)=(1-(1-\beta)\lambda) I(\phi)-\log \int_W \Omega_{h_\phi} |s|_{h_\phi}^{2(1-\beta)}. 
\end{equation}
Here  the $I$ functional is defined by 
$$ I(\phi)= -\int_0^1 dt\  \frac{1}{V}\int_W \frac{\p \phi(t)}{\p t}  \omega_{\phi(t)}^n, $$
where $V=\int _W \omega^n$, and $\phi(t)$ is any smooth path in $\H_\infty$ with $\phi(0)=0$ and $\phi(1)=\phi$. Just as the case of smooth K\"ahler potentials, one can check that $I$ is well-defined on $\H_\infty$. Indeed, it is also easy to write down an explicit formula, for example 
$$I(\phi)=-\frac{1}{(n+1)V}\sum_{i=0}^n\int_W \phi \omega^i\wedge \omega_{\phi}^{n-i}.$$ So $\Di$ is well-defined on $\H_\infty$, and  one can check a critical point of $\Di$ satisfies the equation $$\omega_\phi^n=C \Omega_{h_\phi}|s|_{h_\phi}^{2(1-\beta)}, $$
where $C$ is an arbitrary constant. So a critical point of $\Di$ determines a weak conical K\"ahler-Einstein metric, by adding an appropriate constant. 

There is now a well-developed study of the geometry of the space of K\"ahler metrics on a smooth K\"ahler manifold. Part of the theory has been recently adapted to the singular setting, using pluripotential theory. A path $\phi_t$ $(t\in [0,1])$ in  $\H_\infty$ is called a \emph{geodesic} if the function $\Phi$ defined on $W\times [0,1]\times\bR $, viewed naturally as an $n+1$ dimensional variety,  by assigning $\Phi(t, s, x)=\phi_t(x)$, satisfies the homogeneous Monge-Amp\`ere equation: 
$$(\omega+i\p\bp \Phi)^{n+1}=0, $$
in the sense of measures as above. 
 For any two points $\phi_0$ and $\phi_1\in \H_\infty$, there is a unique geodesic  connecting them \cite{Bern}. The reason  for the name ``geodesic" is that when $\phi_t$ is a path of smooth potentials on a smooth K\"ahler manifold, this is the usual geodesic equation for the natural Riemannian metric on $\H$. \\

The following was first proved in \cite{Bern} when $W$ is smooth and later generalized to the singular setting in \cite{BBEGZ}.  For the convenience of readers we provide an exposition of the proof in  Appendix 1. 

\begin{prop} \label{Ding convexity}
$\Di$ is convex along a geodesic in $\H_\infty$. 
\end{prop}

This has an immediate corollary:

\begin{cor} \label{lower bound}
If $\omega_\phi$ is a weak conical K\"ahler-Einstein metric on $(W, \Delta,\beta)$, then $\Di$ achieves its minimum at $\phi$ on $\H_\infty$. 
\end{cor}

Now we start the proof of Theorem  \ref{reductive}.  Recall that we assume  $W$ is embedded into $\bC\bP^N$ by sections of the line bundle $L=K_W^{-m}$. Any element in $G=\Aut(W, \Delta)$ acts naturally  on $H^0(W, K_W^{-m})$, so is induced by an element in $PGL(N+1;\bC)$. Then it is easy to see that $G$ is the closed subgroup of $PGL(N+1;\bC)$ consisting of those elements that preserves $W$ and $\Delta$. The Lie algebra $Lie(G)$ can be identified with the space of holomorphic vector fields on $\bC\bP^N$ that is tangential to $W$ and $\Delta^{(0)}$. 

Let $K$ be the subgroup of $G$ that preserves the weak conical K\"ahler-Einstein metric $\omega_\phi$ on $W_0\setminus \text{supp} \Delta^{(0)}$. Then any element $k$ in $K$ lifts to an action on $L$ that preserves the metric $h_\phi$. To see this, choose an arbitrary lift of $k$ to an action on $L$ - this is possible by previous discussion. Since $k^*h_\phi$ and $h_\phi$ give rise to the same K\"ahler metric $\omega_\phi$, the equation
$i\p\bp \log(k^*h_\phi/h_\phi)=0, $
holds on $W_0\setminus \text{supp} \Delta^{(0)}$.  On the other hand, by definition $\log (k^*h_\phi/h_\phi)$ is bounded, so it is a constant since $W$ is normal. Then by multiplying by a constant we can assume $k^*h_\phi=h_\phi$.
It follows that $K$ is a closed subgroup of $G$. We furthermore claim it is compact. To see this, notice that since $\phi$ is bounded, the Hermitian metric $h_\phi$ on $L$ is uniformly equivalent to the metric $h_0$ defined by $m^{-1}\omega_{FS}$. Moreover, since $\omega_\phi$ is smooth on $W_0\setminus \text{supp}\Delta_0$ and has finite Monge-Amp\`ere measure, one easily sees that it defines a Hermitian metric on $H^0(W, L)$.  We choose  an orthonormal basis $\{s_i\}$ of $H^0(W, L)$  with respect to the metric defined by $\omega_\phi$. Then $\{s_i\}$ induces another embedding of $W$ into $\bC\bP^N$, and this realizes $K$ as a closed subgroup of $PU(N+1;\bC)$. In particular, $K$ is compact. 

Now it is easy to see that the Lie algebra $Lie(K)$ consists of holomorphic Killing vector fields on $(W_0, \omega_\phi)$ tangential to $\Delta^{(0)}$. Any element in $Lie(G)$ is generated by a  bounded complex-valued Hamiltonian function on $W_0\setminus \text{supp}\Delta^{(0)}$ with respect to the metric $\omega_\phi$, up to a constant. From this point of view,  elements in $Lie(K)$ corresponds exactly to those Hamiltonian functions that are real-valued. 

The following  is proved in \cite{Bern} and \cite{BBEGZ}, to generalize the Bando-Mabuchi uniqueness theorem. The proof depends on carefully looking at  the case when the Ding functional is constant along a geodesic. We will outline the proof in Appendix 1.

\begin{prop} \label{uniqueness} Given another weak K\"ahler-Einstein metric $\omega_{\phi'}$ on $(W, \Delta, \beta)$, there are a geodesic $\phi(t)$ $(t\in [0,1])$ in $\H_{\infty}$ connecting $\phi$ and $\phi'
$, and a holomorphic vector field $Y$ that preserves $\Delta$ such that $JY\in Lie(K)$, and $f_t^*\omega_{\phi(t)}=\omega_\phi$, where $f_t$ is the family of holomorphic transformations on $W$ generated by $Y$.
\end{prop}

Now we continue the proof of Theorem \ref{reductive}.
  We denote by $K^c\subset G$ the complexification of $K$ (defined in the Lie algebra level). It suffices to prove $G=K^c$. For any $g\in G$, we lift the action to $L$, so we view it as an element in $GL(N+1;\bC)$. Consider the pull back metric $g^*h_\phi$ on $L$, and the corresponding weak K\"ahler metric $\omega_{\phi'}=g^*\omega_\phi$. Then it is clear that $\omega_{\phi'}$ also satisfies the Equation (\ref{weak conic KE}) on $W_0\setminus \text{supp} \Delta^{(0)}$. Now since $g$ is a projective transformation, we may write $g^*\omega_{FS}=\omega_{FS}+\sqrt{-1}\p\bp \log |g.x|^2$. Thus $\phi'=\phi+m^{-1}\log |g.x|^2+\text{constant}\in \H_\infty$. Therefore  $\omega_{\phi'}$ is also a weak conical K\"ahler-Einstein metric. By Proposition  \ref{uniqueness}, there is a holomorphic vector field $Y\in Lie(K^c)$ such that the time one flow $f_1$ of $Y$ satisfies $f_1^*\omega_{\phi'}=\omega_\phi$.  Now $g\circ f_1\in G$  preserves $\omega_\phi$, so $g\circ f_1\in K$, thus $g\in K^c$. Therefore $G=K^c$ is reductive.\\

Now we move on to Futaki invariant.\\
 
 Suppose there is a $\bC^*$ action on a $\bQ$-Fano variety $W$ and the action lifts to the ample line bundle $L=K_{W}^{-m}$.   
For integers $k\geq 1$ we have a vector space $H^{0}(W, L^{k})$ with a $\bC^{*}$-action. Let $d_{k}$ be the dimension of this vector space and $w_{k}$ be the total weight of the action. By general theory these are, for large $k$, given by polynomials in $k$ of degrees $n, n+1$ respectively. Thus 
   \begin{equation} \label{Futaki definition} \frac{w_{k}}{k d_{k}} = F_{0} + F_{1} k^{-1} + O(k^{-2})\end{equation}
   and we define the Futaki invariant of the $\bC^*$ action to be $\Fut(W)=F_{1}$. This definition also generalises to the KLT pair $(W, (1-\beta)\Delta)$, when the $\bC^*$ action preserves $\Delta$. Then we have an action on $H^0(\Delta, L^k)$, and we define
   
   \begin{equation} \label{modified Futaki}
   \Fut_\beta(W, \Delta)=\Fut(W)+(1-\beta)c(F_0-F_0'), 
   \end{equation}
where $F_0'$ is defined in the same way as in Equation (\ref{Futaki definition}), but with $W$ replaced by $\Delta$, and $c=\frac{n\lambda}{2}$. 

Following \cite{Do11} we give an analytic definition of the Futaki invariant, in terms of a smooth K\"ahler metric. 
Let $\omega$ be the restriction of the Fubini-Study metric under an embedding $W\hookrightarrow \bP^N$ induced by $K_W^{-m}$.  The $S^1$ action is generated by a Hamiltonian function $u$. Then we define
$$Fut_\beta(W, \Delta)=-\int_W u(Ric(\omega)-\omega)\wedge\omega^{n-1}+(1-\beta)\left(2\pi
\int_{ \Delta}u \omega^{n-1}-
\lambda\int_W u\omega^n\right).$$
It follows from the arguments in \cite{DT} that this is well-defined  and  does not depend on the choice of $\omega\in 2\pi c_1(W)$.
To see this agrees with the previous algebraic definition of Futaki invariant (up to a constant multiple),  we take an equivariant log resolution $p: (W', \Delta')\rightarrow(W, \Delta)$. Then we simply need to calculate the algebraic Futaki invariant on $(W', \Delta')$ with respect to the polarization $p^*(-mK_W)$. One can use the (equivariant) Riemann-Roch formula (as in \cite{Do02}) to compute it using a smooth background metric.
On the other hand, a smooth metric $\omega$ on $W$ pulls back to a smooth form $\omega'$ on $W'$ which is degenerate on the exceptional divisors. But $\omega'$ still lies in the correct cohomology class and it is then straightforward to check that one can use $\omega'$ to do the same computation, and gives rise to the above analytic formula.

\begin{thm} \label{Futaki}
If $(W, \Delta, \beta)$ admits a weak conical K\"ahler-Einstein metric, then the Futaki invariant $\Fut_\beta(W, \Delta)=0$ with respect to any one parameter subgroup of $\Aut(W, \Delta)$. 
\end{thm}

This follows from a general result  of Berman \cite{Ber2}, but what we need is much easier, and we provide a direct proof here, using the analytic definition of the Futaki invariant.   Let $f_t$ be the real one parameter subgroup of the $\bC^*$ action.  Write $f_t^*\omega=\omega+\sqrt{-1}\p\bp\phi(t)$, then $\dot{\phi}$ is the Hamiltonian function generating the $S^1$ action. We have

\begin{lem} \label{Futaki formula}
$$Fut_\beta(W, \Delta)=\frac{1}{n}\frac{d}{dt} \Di(\phi(t))$$
\end{lem}

Since by Corollary \ref{lower bound} the Ding functional $\Di$ is bounded from below on $\H_\infty$, it follows that $Fut_\beta(W, \Delta)=0.$ This proves Theorem \ref{Futaki}.\\

Now we prove Lemma \ref{Futaki formula}. Let $h$ be the Hermitian metric on $K_W^{-1}$ corresponding to $\omega$. 
By pulling back to the log resolution it is easy to see that  the following quantity 
$$Q(\omega)=\int_W \log \frac{\omega^n}{\Omega_{h}|s|_{h}^{2\beta-2} (\int_W\Omega_{h}|s|_{h}^{2\beta-2})^{-1}}\omega^n$$
is well-defined. It also follows from the definition that
$$Q(f_t^*\omega)=Q(\omega). $$
  Taking derivatives with respect to $t$ we obtain
\begin{eqnarray*}
&&\int_W \dot{\phi}(-Ric(\omega)+\omega+(1-\beta)\sqrt{-1}\p\bp\log |s|^2_{h})\wedge \omega^{n-1}\\&+&\frac{d}{dt} \log \int_W \Omega_{h}|s|_{h}^{2\beta-2}+(1-(1-\beta)\lambda)\int_W \dot{\phi}\omega^{n}=0.
\end{eqnarray*}
Thus $$\frac{d}{dt}\Di(\phi(t))=n\int_W \dot{\phi} (-Ric(\omega)+\omega)+n(1-\beta)[2\pi \int_\Delta \dot\phi \omega^{n-1}-\lambda\int_W \dot\phi\omega^n].$$\\

Theorem \ref{thm4} is a combination of Theorem \ref{reductive} and Theorem \ref{Futaki}. 

\section{Completion of Proof of Theorem \ref{main theo}}
We recall the definition of K-stability from \cite{CDS0}.
\begin{defn} Let $X$ be an $n$-dimensional Fano manifold. A test-configuration for $X$ is a flat family $\pi:{\mathcal X}\rightarrow \bC$ embedded in $\bC\bP^{N}\times \bC$ for some $N$, invariant under a $\bC^{*}$ action on $\bC\bP^{N}\times \bC$ covering the standard action on $\bC$ such that 
\begin{itemize}
\item $\pi^{-1}(1)=X$ and the embedding $X\subset \bC\bP^{N}$ is defined by the complete linear system $\vert -r K_{X}\vert $ for some $r$;
\item The central fibre $X_{0}=\pi^{-1}(0) $ is a normal variety with log terminal singularities.
\end{itemize}
\end{defn}

A test configuration has a basic numerical invariant: the {\it Futaki invariant}. This is defined to be the Futaki invariant of the central fiber, with respect to the induced $\bC^*$ action (see Equation (\ref{Futaki definition})). 
   
   \begin{defn} $X$ is K-stable if for all non-trivial test configurations ${\mathcal X}$ we have $\Fut({\mathcal X})\geq 0$ and strict inequality holds if ${\mathcal X}$ is non-trivial.
\end{defn}

By \lq\lq non-trivial'' here we mean that the central fibre is not isomorphic to $X$. The condition defined above is  often called \lq\lq polystability'' in the literature.
Recall that our main result (Theorem \ref{main theo}) states that if a Fano manifold is K-stable then it admits a K\"ahler-Einstein metric. Converse results, in different degrees of generality have been proved by have been proved by Tian \cite{Ti1}, Stoppa \cite{Stoppa} and Berman \cite{Ber2}. The sharp form proved by Berman shows that in fact K-stability (as we have defined it) is equivalent to the existence of a K\"ahler-Einstein metric. 
  
  In this section we will show how to deduce Theorem \ref{main theo} from the results we obtained in \cite{CDS1}, \cite{CDS2} and this paper. Fix a $K$-stable Fano manifold $X$. \\

\noindent {\bf Step 1:} Choose an integer $\lambda>0$ and a smooth divisor $D$ in the linear system $|-\lambda K_X|$. This is possible if we choose $\lambda$ sufficiently large, by Bertini's theorem. In practice $\lambda$ does not need to be big, but there are cases where we can not choose $\lambda=1$. But for the simplicity of our argument we always choose $\lambda>1$. Then we consider K\"ahler-Einstein metrics on $X$ with cone angle $2\pi\beta$ along $D$:
$$Ric(\omega_\beta)=(1-(1-\beta)\lambda) \omega_\beta+2\pi(1-\beta) [D].$$
More precisely, we use the definition recalled in \cite{CDS2}. This means that we require the potential function of $\omega_\beta$ to lie in $C^{2, \alpha, \beta}$ for all $\alpha\in (0, \beta^{-1}-1)$, in the sense defined in \cite{Do11}, and satisfies the K\"ahler-Einstein equation on $X\setminus D$.  When $\beta=1$, it follows from standard elliptic regularity that this means $\omega_1$ is a smooth K\"ahler-Einstein metric on $X$, where $D$ disappears in the definition. \\

\noindent {\bf Step 2:}  Define $I$ to be the set consisting of all $\beta\in (0,1]$ such that there exists a K\"ahler-Einstein metric on $X$ with cone angle $2\pi\beta$ along $D$, in the sense defined above. Our goal is to prove $``1"\in I$. \\

\noindent {\bf Step 3:} We first prove $I$ is nonempty.  Choose an integer $N\geq \frac{\lambda}{\lambda-1}$, and let $\beta_0=1/N$. Then a K\"ahler-Einstein metric on $X$ with cone angle $2\pi\beta_0$ along 
$D$ is the same as a smooth K\"ahler-Einstein metric on a K\"ahler orbifold $\hat{X}$ with Ricci curvature $1-(1-\beta)\lambda\leq0$. We briefly recall the construction of $\hat X$. We cover $D$ by coordinate charts $(U_{\alpha}, z_{\alpha})$ such that $U_\alpha=\{|z^1_\alpha|\leq 1, \cdots, |z^n_\alpha|\leq 1\}$ and  $U_\alpha\cap D=\{z_\alpha^1=0\}$.  Then replacing the coordinate function $(z_\alpha^1, z_\alpha^2, \cdots, z_\alpha^n)$ by $((z_\alpha^1)^{1/N}, z_\alpha^2, \cdots, z_\alpha^n)$, we define an orbifold chart on $U_\alpha$. It is not difficult to see that these patch together and form an K\"ahler orbifold, which we denote by $\hat{X}$. It is then straightforward to check that a smooth K\"ahler metric on $\hat{X}$ is a K\"ahler metric on $X$ with cone angle $2\pi\beta_0$ along $D$, in the above sense. The existence of  a smooth K\"ahler-Einstein metric on $\hat X$ follows easily from the classical theory of Aubin and Yau, extended to the orbifold setting. The latter is well-known, see for example \cite{DK} -the key point is that the maximum principle works readily on an  orbifold. We outline some details here. Choose a smooth background  metric $\omega$ on $\hat X$, then $Ric(\omega)=(1-(1-\beta))\lambda\omega+i\p\bp f$ for some smooth function $f$ on $\hat X$.  We consider the continuity  equation 
$$(\omega+i\p\bp \phi_t)^n=e^{-t(1-(1-\beta)\lambda)\phi_t+f}\omega^n. $$
When $t=0$, the equation has a trivial solution $\phi_t=0$;  when $t=1$ a solution gives rise to  the desired K\"ahler-Einstein metric. 
The $C^0$ estimate follows from a straightforward maximum principle (applied on the orbifold cover near the maximum point).  The $C^2$ estimate also follows from the maximum principle and the fact the curvature of $\omega$ is bounded (since $\omega$ is smooth). Then we can apply Calabi's third derivative estimate  to obtain higher regularity.  In short we have shown $\beta_0\in I$. In particular, $I$ is non-empty. \\

\noindent {\bf Step 4:} $I$ is open in $(0, 1)$. Based on a linear elliptic estimate for K\"ahler metrics with cone singularities, this is proved in \cite{Do11}, under an additional assumption that there is no nonzero holomorphic vector field on $X$ which is tangential to $D$. The latter is confirmed in  \cite{SW}. \\

\noindent {\bf Step 5:} To deal with closedness we need to use the assumption on  $K$-stability.  If ${\mathcal X}$ is a test configuration for $X$, as above then we can extend $D\subset X=\pi^{-1}(1)$ to a divisor in ${\mathcal X}$ and obtain a $\bC^{*}$-invariant divisor $D_{0}\subset X_{0}$. Then we define the modified Futaki invariant
$$  \Fut_{\beta}({\mathcal X})= \Fut_\beta(X_0, D_0), $$
where the latter is given by (\ref{modified Futaki}). 
A clear  but  important fact is that the Futaki invariant is linear in $\beta$. One can define similarly the notion of $K$-stability for $(X, D, \beta)$, in the most obvious way. Then the $K$-stability is also linear in $\beta$. 
\begin{lem}
$(X, D, 0)$ is $K$-semistable, i.e. for any test configuration $\mathcal X$ for $X$, we have $\Fut_0(\mathcal X)\geq 0$. 
\end{lem}
This is proved in \cite{Sun} when $\lambda=1$. The argument makes use of the Calabi-Yau metric on $D$ to construct almost balanced embeddings of $(X, D)$.  It also extends easily to the general case $\lambda>1$ with the Calabi-Yau metric replaced by a K\"ahler-Einstein metric on $D$ (which exists because $c_1(D)<0$). This Lemma also follows from the more general results \cite{Ber2}, \cite{LS}, \cite{OS}. 
Since we assume $X$ is $K$-stable, it follows  that $(X, D, \beta)$ is $K$-stable for any $\beta\in (0, 1]$. \\

\noindent {\bf Step 6:} Now it suffices to show that if  $\beta_i\in I\cap [\beta_0, 1]$ increases to $\beta_\infty$, then $\beta_\infty\in I$. We choose  a corresponding sequence of K\"ahler-Einstein metrics $\omega_{\beta_i}$.  There are two cases:

\begin{enumerate} 
\item $\beta_\infty\leq1-\lambda^{-1}$. In this case the Ricci curvature of  $\omega_{\beta_i}$ is non-positive, and the $K$-stability of $X$ is not required. A general result for the existence of K\"ahler-Einstein metrics with cone singularities in the case of non-positive curvature can be found in \cite{JMR}.  For our convenience we give an alternative argument, in the line of our series of papers. Notice that in \cite{CDS1} and \cite{CDS2} our main interest was in the opposite case when $\beta_\infty>1-\lambda^{-1}$, but it is not hard to see that  the bulk of the discussion in \cite{CDS1} and \cite{CDS2} holds in the general case, as long as we assume an extra volume non-collapsing condition, or equivalently, a diameter bound.  By Theorem 3.10 in \cite{CDS1} this condition is indeed satisfied in our setting so the results in \cite{CDS2} do apply.  Choose a sequence of background K\"ahler metrics $\omega_{\beta_i}'$ with a limit $\omega_{\beta_\infty}'$.  Write $\omega_{\beta_i}=\omega_{\beta_i}'+i\p\bp \varphi_i$. Then,  by the proof of Theorem 3.10 in \cite{CDS1} there is an a priori $L^\infty$ bound on $\varphi_i$. Then we can apply results in Section 3.2 of \cite{CDS2} to conclude that $\omega_{\beta_i}$ converges to a limit $\omega_{\beta_\infty}$, which is a K\"ahler-Einstein metric on $X$ with cone angle $2\pi\beta_\infty$ along $D$. Thus $\beta_\infty\in I$.

\item $\beta_\infty>1-\lambda^{-1}.$ Then by Theorem 1 in our previous paper \cite{CDS2} and Theorem 1 in this paper, we obtain a Gromov-Hausdorff limit $W$ which is a $\bQ$-Fano variety, and a Weil divisor $\Delta$ (if $\beta_\infty=1$ then $\Delta$ is empty), such that $(W, \Delta, (1-\beta_\infty))$  is KLT, and admits a weak conical K\"ahler-Einstein metric.  Moreover, the limit can be taken in a fixed projective space so we can view it as a limit of Chow cycles.  By Theorem \ref{reductive}  we know $\Aut(W, \Delta)$ is reductive. So by general theory of Luna slices (see  \cite{Do10}) one can construct a test configuration $\mathcal X$ for $X$ with central fiber $W$ so that $\Delta$ is the flat limit of $D$. By Theorem \ref{Futaki}, $\Fut_{\beta_\infty}(\mathcal X)=0$. By our definition of $K$-stability and Step 5, we conclude that $(W, \Delta)$ is isomorphic to $(X, D)$. Then it follows from Theorem 2 in \cite{CDS2} and Theorem \ref{thm1} in this paper that $\beta_\infty\in I$. 
\end{enumerate}

Therefore we conclude that $X$ admits a K\"ahler-Einstein metric. This finishes the proof of Theorem  \ref{main theo}.

\section{Appendix 1: Convexity of Ding functional and uniqueness of weak conical K\"ahler-Einstein metrics}
We reproduce the proof of Proposition  \ref{Ding convexity} and Proposition  \ref{uniqueness}, mainly following \cite{Bern} (compare also \cite{BBEGZ}).  A more detailed account of these can be found in the expository note by Long Li \cite{LL}. \\

 Take a log resolution $p: W' \rightarrow (W, \Delta)$. By definition  we have
$$-K_{W'}=-p^*(K_W+(1-\beta)\Delta)-E+\Delta', $$
where $E$ and $\Delta'$ are effective divisors with normal crossing intersections, $\Delta'$ has  coefficients in $(0,1)$, $E$ has integer coefficients, and $p_*(\Delta'-E)=\Delta$.  We denote by $L$ the line bundle $K_{W'}^{-1}\otimes E$ on $W'$.   
Then $L$ is linearly equivalent to $-p^*(K_W+(1-\beta)\Delta)+\Delta'$. \\

Now we fix a smooth metric $h_0$ on $K_W^{-1}$, in the sense defined before, then $h_0$ determines a smooth metric  on $-p^*(K_W+(1-\beta)\Delta)$, which we still denote by $h_0$. We also fix a smooth K\"ahler metric  $\omega'$ on $W'$ and a smooth metric $h_0'$ on $L$. Write $\Delta'=\sum_i a_i E_i$, and choose a defining section $s_i$ of $E_i$, then $s=\otimes_i s_i^{\otimes a_i}$ is formally a defining section of $\Delta'$, so it defines a Hermitian metric $k_0$ on the $\bR$-line bundle $L_{\Delta'}$ with curvature $2\pi [\Delta']$, i.e. 
  $$i\p\bp \log |s|^2=i\p\bp \log \Pi_i |s_i|^{2a_i}=2\pi \sum_i a_i [\Delta_i]=2\pi [\Delta']. $$
With these fixed, then any metric $h_\phi=h_0e^{-\phi}$ on $K_W^{-1}$ defines a metric  $h_0 e^{-\phi} \otimes k_0$ on $L$, which for simplicity we denote by $h'_{\widetilde \phi}=h'_0e^{-\widetilde \phi}$. Here as before we have used the identification between plurisubharmonic functions on $W$ and $W'$. \\

Recall for any $\phi \in \H_\infty$, we have defined the Ding functional (c.f. Equation (\ref{definition Ding functional})). We denote the second term by  
$$\widetilde{\Di}(\phi)=-\log \int_W \Omega_{h_\phi} |s|_{h_\phi}^{2(1-\beta)}. $$
Then it is not hard to express this as an integral on $W'$. Notice by definition there is an $L$-valued $n$-form $u$ on $W'$(unique up to a constant multiple) with zero divisor $E$. Then we have 
$$\widetilde{\Di}(\phi)=\widetilde{\Di}(\widetilde \phi):=-\log \int_{W'} a_n  u\wedge \bar u. $$
where $a_n=(-1)^{\frac{n(n-1)}{2}}i^n$, and we have used the metric $h'_{\widetilde \phi}$ on $L$. To be more precise this holds up to an additive constant, but we always fix this constant to be zero by a possibly different choice of $u$. 

We now compute the second derivative of $\widetilde{\Di}$. By Hodge theory for any smooth metric $h'_{\psi}=h'_0e^{-\psi}$ on $L$, and for any $L$-valued $(n, 0)$ form $\xi$ that is orthogonal to the space of holomorphic sections,  there is an $L$-valued $(n, 1)$ form $\alpha$ such that 
$$\bp_\psi^*\alpha=\xi, $$
where we have used the Hermitian metrics defined by the metric $\omega'$ and $h'_\psi$.  Moreover we choose $\alpha$ to be orthogonal to the kernel of $\bp^*_{\psi}$ so it is uniquely determined by $\xi$. Write $\alpha=v\wedge \omega'$, then $\bp v\wedge\omega'=0$, and up to a factor we have $\p_\psi v=\xi$, where $\p_\psi$ is the $\p$ operator on forms coupled with the connection on $L$ defined by $h_{\psi}'$. 

 Given a smooth family of smooth metrics  $h'_{\psi_t}$  on $L$ for $t\in R=\{t\in \bC|Re(t)\in [0,1]\}$, we  solve 
 $$\p_{\psi_t} v=P_t(\frac{\p \psi}{\p t} u), $$
 where $P_t$ is the projection to the orthogonal complement of holomorphic sections (again we use the metric defined by $h'_{\psi_t}$ and $\omega'$). On $W\times R$, we denote $\hat{u}=u-dt\wedge v$, and let $\omega_\psi$ be the curvature form of $h'_{\psi}$ on $L$ (naturally pulled back from $W$). From now on, we use the notation $||\cdot||$ to denote the $L^2$ norm with respect to the K\"ahler metric $\omega'$ and the Hermitian metric written in the subscript. 
On $R$, we have (for simplicity of notation we omit the variable $t$)

\begin{lem}[\cite{Bern}, Theorem 3.1]  \label{positive curvature}
$$||u||_{\psi}^2 i\p\bp \widetilde{\Di}(\psi(t))=a_n\int_{W'}\omega_{\psi} \wedge \hat{u}\wedge \bar{\hat{u}}+||\bp v||_{\psi}^2 i dt\wedge d\bar t,  $$ 
\end{lem}
 This follows from the general positivity of direct image bundles discovered by Berndtsson. For our convenience we present a  direct calculation here. First a straightforward calculation gives
\begin{eqnarray*}
 ||u||^2_{\psi}\p_t\bp_t \widetilde{\Di}(\psi(t))&=&a_n\int_{W'} \frac{\p^2 \psi}{\p t\p \bar t}u\wedge \bar u-(a_n\int_{W'} |\frac{\p\psi}{\p t}|^2 u\wedge \bar u
       -\frac{|\int_{W'}\frac{\p \psi}{\p t}u \wedge \bar u|^2}{||u||_{\psi}^2}) 
 \end{eqnarray*}
Now we compute 
\begin{eqnarray*}
&&a_n\int_{W'} \omega_\psi \wedge \hat{u}\wedge \bar{\hat{u}}\\
&=& a_n(i \int_{W'}\frac{\p^2 \psi}{\p t\p \bar t}u\wedge \bar u+i\int_{W'}(-1)^n  \bp \frac{\p\psi}{\p t} \wedge u\wedge \bar v+i\int_{W'} \p \frac{\p \psi}{\p \bar t} \wedge v\wedge \bar u
\\&&+\int_{W'}(-1)^{n-1} \omega_\psi \wedge v\wedge \bar v) dt\wedge d\bar t\\
&=&(I+II+III+IV)idt d\bar t.
\end{eqnarray*}
Since $\p_{\psi} v= P_t(\frac{\p\psi}{\p t} u)$ we have 
$$\bp \p_{\psi} v=\bp\frac{\p\psi}{\p t}\wedge u. $$
So 
$$II=(-1)^na_n\int_{W'} \bp\frac{\p\psi}{\p t}\wedge u\wedge \bar v=-a_n \int_{W'}\p_{\psi} v\wedge\bp_{\psi} \bar v=-||\p_{\psi} v||_{\psi}^2, $$
and 
$$III=a_n\int_{W'} \p \frac{\p\psi}{\p \bar t}\wedge v\wedge \bar u=(-1)^{n-1}a_n\int_{W'}v \wedge \overline{\bp\p_{\psi} v}=-||\p_\psi v||_{\psi}^2. $$
Now we have the Bochner formula 
$$i\bp \p_{\psi} v+i\p_{\psi} \bp v=\omega_{\psi} \wedge v. $$
So by integration by parts we obtain
\begin{eqnarray*}
IV&=&
a_n\int_{W'}\p_{\psi}  v\wedge \bp_{\psi} \bar v+a_n\int_{W'} \bp  v\wedge \p \bar v\\
&=& ||\p_{\psi} v||_{\psi}^2-||\bp v||_{\psi}^2. 
\end{eqnarray*}
Here we used the fact that $\bp v\wedge \omega'=0$. 
Thus we have
$$II+III+IV=-||\p_{\psi} v||_{\psi}^2-||\bp v||_{\psi}^2. $$
Now by definition
$$||\p_{\psi}  v||_{\psi}^2=||\frac{\p\psi}{\p t} u||_{\psi}^2-\frac{(\int_{W'}\frac{\p\psi}{\p t} u\wedge \bar u)^2}{||u||_{\psi}^2}. $$
Combining all the above together this finishes the proof of the Lemma.\\

Now recall a geodesic in $\H_\infty$ is a path $\phi_t$ in $\H_\infty$ such that the map $\Phi: W\times R \rightarrow \bR$ defined by $\Phi(x, t)=\phi_{Re(t)}(x)$ satisfies $i\p\bp \Phi\geq0$ and the homogeneous Monge-Amp\`ere  equation 
\begin{equation} \label{geodesic equation}
(i\p\bp \Phi)^{n+1}=0.
\end{equation}

\begin{lem} [\cite{Bern}, Section 2.2]
 Given $\phi_0, \phi_1\in\H_\infty$, there is a unique  bounded geodesic $\phi_t$ connecting them. Moreover, $\phi_t$ is Lipschitz in $t$, i.e. there is a constant $C>0$ depending only on $\phi_0$ and $\phi_1$, such that $|\phi_t-\phi_s|_{L^\infty}\leq C|t-s|$. 
\end{lem}

The proof uses methods of barrier functions, and characterization of weak solutions to the homogeneous complex Monge-Amp\`ere equations by maximal functions. On a smooth K\"ahler manifold, one obtains stronger regularity for the geodesic equation through more delicate PDE method, see \cite{Chen}. \\

To prove Proposition  \ref{Ding convexity}, we fix $\phi_0$, $\phi_1$ and the geodesic $\phi_t=\Phi(\cdot, t)$.  It is easy to see that $I(\phi_t)$ is a linear function of $t$, so it suffices to prove the convexity of $\widetilde{\Di}(\widetilde{\phi}(t))$. For simplicity of notation we denote $\psi(t)=\widetilde{\phi}(t)$. We first approximate $\psi$ be a sequence of smooth functions with curvature control, then apply Lemma \ref{positive curvature}, and take limit. \\
  
  We first approximate the metric $k_0$ on $L_{\Delta'}$. 
  Fix a smooth metric $k$ on $L_{\Delta'}$. Then for $\epsilon>0$ we define a new metric $k_\epsilon$ by 
  $$||e||_{k_\epsilon}^2=\frac{||e||_{k_0}^2||e||^2_k}{||e||^2_k+\epsilon ||e||^2_{k_0}}.$$
  This is a smooth metric on $L_{\Delta'}$ for $\epsilon>0$. Moreover on $W'\setminus \text{supp}\Delta'$, we have
  \begin{eqnarray*}
  -i\p\bp \log k_\epsilon&=&-i\p\bp \log k+i\p\bp \log (||e||^2_k+\epsilon ||e||^2_{k_0})\\
  &=&-i\p\bp \log k+i\p\bp \log (||s||^2_k+\epsilon)\\
  &\geq&-i\p\bp\log k+\frac{||s||^2_k}{||s||^2_k+\epsilon}i\p\bp\log k\\
  &\geq&-\frac{\epsilon}{||s||_k^2 +\epsilon} i\p\bp \log k.
  \end{eqnarray*}
  Write $k_\epsilon=ke^{-\tau_\epsilon}$, $k_0=ke^{-\tau_0}$,  and $\omega_{\tau_\epsilon}$ the curvature form of $k_\epsilon$. Then
  we have obtained  (Compare \cite{CDS1}, Proposition 2.1)
  \begin{lem}[\cite{Bern}, Section 2.3] \label{approximation singular metric}
 As $\epsilon$ tends to zero,  $\tau_\epsilon$ decreases to $\tau_0$,  with $\omega_{\tau_\epsilon}\geq -C\omega'$ for a uniform $C$. Furthermore, outside a fixed neighborhood $U$ of $\text{supp}\Delta'$, there is a constant $C_U$ so that  $\omega_{\tau_\epsilon}\geq-\epsilon C_U \omega'$.
   \end{lem}

Next we approximate the metric $h_\phi$ on $-p^*(K_W+(1-\beta)\Delta)$. We denote by $\pi_1$ the projection map $W'\times R\rightarrow W'$, and $\Omega$ the form $\pi_1^* \omega'+idtd\bar t$ on $W'\times R$. Since $K_W^{-1}$ is ample, $\pi_1^*p^*K_W^{-1}$ admits a metric with non-negative curvature. So  by a general result of Blocki-Kolodziej (see \cite{Bern}, Section 2.3) one can approximate $\phi$ by a decreasing sequence of smooth functions $\phi_\epsilon$ with $\omega_{\phi_\epsilon}\geq -\epsilon \Omega$. Moreover we can also assume that $\phi_\epsilon$ depends only on $Re(t)$ (We could work on the product of $W\times R/\bZ$, and take an $S^1$ average). So together with Lemma \ref{approximation singular metric} we have an approximation of the metric $h'_{\psi}$ on $L$ by smooth metrics $h_\epsilon'=h_0'e^{-\phi_\epsilon}\otimes k_\epsilon$.  The curvature of the approximating metric has curvature uniformly bounded from below by $-C\Omega$ on $W\times R$ and bounded below by $-\epsilon C_U\Omega$ outside any neighborhood $U$ of $\text{supp} \Delta'$.  For simplicity of notation we denote $h_\epsilon'=h_0'e^{-\psi_\epsilon}$. Then $\psi_\epsilon$ decreasingly converges to $\psi$.   From Lemma \ref{positive curvature} we get
  \begin{equation} \label{est1}
  \p_t\bp_t \widetilde{\Di}(\psi_\epsilon(t))\geq -\epsilon C_U ||\hat u_\epsilon||^2_{\psi_\epsilon}-C\int_U|v_{\epsilon}|^2_{\psi_\epsilon}+||\bp v_\epsilon||_{\psi_\epsilon}^2. 
  \end{equation}
To proceed  we need a uniform $L^2$ estimate for the equation $\p_{\psi_\epsilon} v_\epsilon=u$.   This is given by the following, using a local H\"ormander estimate.

\begin{lem}[\cite{Bern}, Lemma 6.2] There is a uniform constant $C$ such that 
$$||v_\epsilon||_{\psi_\epsilon}^2\leq C ||\frac{\p \psi_\epsilon}{\p t}u||_{\psi_\epsilon}^2 , $$
\end{lem}

Now we continue the proof of Proposition  \ref{Ding convexity}.  
\begin{lem}There is a constant $C(r)$ depending only on $r>0$ small,  so that for $t\in [r, 1-r]$ and all $\epsilon$ we have 
$$|\dot{\psi}_\epsilon(t)|\leq C(r). $$
\end{lem}
This follows from  the facts $\frac{d^2}{dt^2}\psi_\epsilon(t) \geq -C$,  $\psi_\epsilon$ decreases to $\psi$ and that $\psi$ is Lipschitz in $t$. \\

Since $||u||_{\psi}^2$ is bounded and $\psi_\epsilon$ decreases to $\psi$, by the above two lemmas we know $||v_\epsilon(t)||_{\psi_\epsilon}^2\leq C(r)$ for $t\in [r, 1-r]$. 

\begin{lem}[\cite{Bern} Lemma 6] \label{Lemma 6}
There is a constant $c_\delta$ (independent of $\epsilon$ and $t$) that goes to zero as $\delta$ goes to zero so that  for any $v\in \Omega^{n-1, 0}(L)$, 
$$\int_{U_\delta} |v|_{\psi_\epsilon}^2\leq c_\delta( \int_{W'} |v|^2_{\psi_\epsilon}+|\bp v|^2_{\psi_\epsilon}), $$
where $U_\delta$ is the $\delta$ neighborhood of $\Delta'$. 
\end{lem}

So by choosing $\delta$ small first and then $\epsilon$ small  we get from Equation (\ref{est1}) that for $t\in[r, 1-r]$ that 
$$\frac{d^2}{dt^2} \widetilde{\Di}(\psi_\epsilon)\geq -\epsilon C_{U_\delta} ||\hat{u}_\epsilon||^2_{\psi_\epsilon}-c_\delta C(r) ||v_\epsilon||^2_{\psi_\epsilon}\rightarrow 0 $$
This implies that $\widetilde{\Di}(\psi(t))$ is convex in $[r, 1-r]$. Let $r\rightarrow 0$ we finish the proof of  Proposition  \ref{Ding convexity}. \\

Now we move on to prove Proposition  \ref{uniqueness}.
So we assume $\phi_0$ and $\phi_1$ satisfy the weak conical K\"ahler-Einstein equation. Then  it follows that $\Di(\phi)$ has derivative 0 at $t=0,1$. The convexity then implies that it is a constant, and $\widetilde{\Di}(\phi)$ is a linear function. Since $\widetilde{\Di}(\psi_\epsilon)$ decreasingly converges to $\widetilde{\Di}(\phi)$, we see that the second derivative of  $\widetilde{\Di}(\psi_\epsilon)$ converges to zero uniformly in $[r, 1-r]$ for any $r>0$.  It follows from Equation (\ref{est1}) and Lemma \ref{Lemma 6}  that $||\bp v_\epsilon||_{\psi_\epsilon}^2$ tends to zero uniformly in $[r, 1-r]$.
Then by passing to a subsequence $v_\epsilon$ converges to a limit $v$ weakly in $L^2(W\times R)$, which is holomorphic along the $W$ direction, and  satisfies the equation $\p_{\psi} v= P(\frac{\p \psi}{\p t} u)$ in the weak sense. Taking $\bp$ on both sides and using the Bochner formula $i\bp \p v+i\p \bp v=\omega_{\psi} \wedge v$ again we obtain
$$\omega_{{\psi}_t}\wedge v=\bp \frac{\p \psi}{\p t}\wedge u. $$
We define a holomorphic vector field $Y'$ on $(W'\setminus E) \times R^0$ by 
$$Y'\lrcorner u= v.$$
We have
$$\omega_{{\psi}_t }\wedge v=-\omega_{{\psi}_t }\wedge (Y'\lrcorner u)=(Y'\lrcorner \omega_{{\psi}_t})\wedge u.$$
It follows that on $W'\setminus E$ we have 
$$Y'\ \lrcorner\ \omega_{{\psi}_t}=i\bp \frac{\p \psi}{\p t}. $$
So
 $$\mathcal L_{Y'} \omega_{\psi}=\frac{\p}{\p t}\omega_{\psi}$$
as currents. 

Now from the previous discussion of convergence  the term $\sigma_\epsilon= a_n\int_{W'\times R_r}\omega_{\psi_\epsilon} \wedge \hat{u}_\epsilon\wedge \bar{\hat{u}}_\epsilon$ also converges to zero for any $r>0$, where $R_r=[r, 1-r]\times \bR$.  For any compactly supported  $L$ valued $(n, 0)$ form $\xi$ which is Lipschitz in $t$ and which does not contain the term $dt$,  we have
\begin{eqnarray*}
&&|\int_{W'\times R_r} i\p\bp\psi_\epsilon\wedge \hat{u}_\epsilon\wedge \bar{\xi}|^2 \\&\leq&
\sigma_\epsilon \int_{W'\times R_r} i\p\bp \psi_\epsilon \wedge \xi \wedge \bar{\xi}\leq C \sigma_\epsilon \int_{W'\times R_r} |\p_t \xi|_{\psi_\epsilon}^2 \omega'^n\wedge idt\wedge d\bar t .
\end{eqnarray*} 
By an easy calculation this implies that
$$\lim_{\epsilon\rightarrow 0} \int_{W'\times R_r} (\frac{\p^2}{\p t \p\bar t}\psi_\epsilon -\p(\frac{\p}{\p t}\psi_\epsilon) (Y'))dt\wedge d\bar t\wedge u \wedge \bar{\xi}=0. $$
Using the definition of $v_\epsilon$ it follows that 
$$\lim_{\epsilon\rightarrow 0} \int_{W'\times R_r} idt\wedge d\bar t\wedge \frac{\p v_\epsilon}{\p \bar t} \wedge \overline{\bp \xi}=0. $$
Let $\alpha$ be a compactly supported smooth $(n,1)$ form on $W' \times R$ which does not contain $dt$, and we decompose 
$$\alpha=\bp \xi+\alpha', $$
where $\alpha'$ is orthogonal to $\text{Im}\bp$ for each $t$. We can write $\alpha'=\beta\wedge \omega'$, for a $(n-1, 0)$ form $\beta$.  By definition of $v_\epsilon$ we know $\bp (v_\epsilon\wedge\omega')=0$, so $\bp  \frac{\p v_\epsilon}{\p \bar t} \wedge \omega'=0$.  Now since $H^{n,1}(W', L)=0$ by Kawamata-Viehweg vanishing theorem (see for example \cite{La}, Remark 9.1.23), we know $\frac{\p v_\epsilon}{\p \bar t} \wedge \omega'$ is $\bp$-exact. So 
$$\lim_{\epsilon\rightarrow 0} \int_{W'\times R} dt\wedge v_\epsilon\wedge \overline{\p \alpha'}=0. $$
Now from the decomposition it follows that we may choose $\xi$ to be Lipschitz in $t$, then by the above discussion 
$$\lim_{\epsilon\rightarrow 0} \int_{W'\times R} dt\wedge v_\epsilon\wedge \overline{\p \alpha}=0. $$
Passing to the limit we obtain that 
$$\int_{W'\times R} dt\wedge v\wedge \overline{\p \alpha}=0. $$
It then follows that $\bp_t v=0$ in the weak sense and thus $v$ is holomorphic on $W'\times R$.

Let $Y=p_*(Y')$. Then since $E$ is exceptional, $Y$ is defined on the smooth part $W_0\times R^0$.  Since $W$ is normal we claim for each $t\in R$, 
$Y_t$ extends to a global holomorphic vector field on $W$. To see this, since $Y_t$ is defined on $W_0$, it induces an infinitesimal action on $H^0(W_0, K_{W_0}^{-m})$. Since $W$ is normal, the latter can be identified with $H^0(W, K_{W}^{-m})$. It follows that $Y_t$ is the restriction of a holomorphic vector field on $\bC\bP^N$, so it naturally extends to $W$.    By previous discussion $Y_t$ has bounded $L^2$ norm locally away from $E$ for $t\in R^0$, so the flow generated by $Y_t$ extends continuously to $t=0$ and $t=1$. 
So we obtain a family of holomorphic transformations $f_t$ of $W$ such that  $\omega_{\phi_t}=f_t^*\omega_{\phi_0}$.   In particular we see that $\phi_t$ also satisfies the weak conical K\"ahler-Einstein equation, with $\Delta$ replaced by $\Delta_t=f_t(\Delta)$.  We claim that $f_t(\Delta)=\Delta$. To see this we first notice that since $Y_t\lrcorner\ \omega_{\phi_t}=i\bp \dot{\phi_t}$,  the imaginary part $Im(Y_t)$ is a Killing vector field for $\omega_{\phi_t}$, so by the K\"ahler-Eistein equation (\ref{weak conic KE}) $Im(Y_t)$ is tangential to $\Delta_t$. Thus $Y_t$ is also tangential to $\Delta_t$. For any smooth $(2n-2)$ form $\eta$ on $W$, we have
$$\frac{d}{dt}\int_{\Delta_t}\eta=\frac{d}{dt}\int_{\Delta} f_t^*\eta=\int_{\Delta_t} d(Re(Y_t)\lrcorner \eta)+\int_{\Delta_t} Re(Y_t)\lrcorner\ d\eta=0, $$
where the first term vanishes because $\Delta_t$ is a closed current, and the second term vanishes because $Y_t$ is tangential to $\Delta_t$. So we conclude that $\Delta_t=\Delta$. This in particular implies that $\phi_t$ is smooth on $W_0\setminus \text{supp}\Delta$, and satisfies the geodesic equation pointwisely there. It is then straightforward to verify that $f_t^*Y_t=Y_0$ with $JY_0\in Lie(K)$,  and $f_t$ is the one parameter subgroup generated by $Y_0$. This finishes the proof of Proposition  \ref{uniqueness}. 

\section{Appendix 2: Proof of Proposition \ref{prop6}}

Adopting the notations in Section 2.5, and we write $\omega_\infty=i\p\bp u$ on $0.8B^{2n}$.   Suppose that
$u$ is a $W^{2, p}$ weak solution to the K\"ahler-Einstein equation in $0.8B^{2n}$
\begin{equation}
\det(u_{i\bar j}) = e^{-\lambda u}.
\label{KE1}
\end{equation}
Moreover, $u$ satisfies the uniform bound $C^{-1}\leq (u_{i\bar j})\leq C$ away from the limit divisor $D_\infty$
(the bound is independent of distance to divisor).  The following proposition  is likely well known to experts,
following a theorem of Trudinger \cite{Tru84}, which
is an extension of the Evans-Krylov theory.  For convenience of
readers, we include a proof here.
\begin{prop} \label{prop9} $u$ extends to a $C^{2, \alpha}$ function on $0.7B$ for some $\alpha>0$, i.e. there exists a constant $C'$  such that
\[
   [D^{2} u]_{C^{\alpha} (0.7B)} \leq C'.
\]
\end{prop}

\begin{proof} 
In local coordinate chart, set
\[
h =  \det (u_{i\bar j}) = e^{-\lambda u}.
\]
Then for a constant $C_1>0$,
\[
   |(\log h)_{i\bar j}| \leq C_1.
\]
For any  unit vector $v\in \mathbb{R}^{2n}$ and any positive small constant $\epsilon > 0$,  and for any function $f$, we can define
 the second difference quotient function
\begin{equation*}
D_{2,\epsilon, v} f (z) =  \epsilon^{-2} \cdot \left( f(z+ \epsilon v) + f(z-\epsilon v) - 2 f(z)\right).
\end{equation*}
We denote $w^\epsilon_v(z)=D_{2, \epsilon, v}u(z)$ and $w_v=\lim_ {\epsilon\rightarrow0}w^\epsilon_v(z)$ when the limit exists. 
In terms of the coordinates $
 z_i = x_i + \sqrt{-1} y_i$, if we choose $v_1=\frac{\p}{\p x_i}, v_2=\frac{\p}{\p y_i}$, we have
\begin{equation}\label{E-1}\lim_{\epsilon\rightarrow 0} \left(w_{v_1}^\epsilon+w_{v_2}^\epsilon \right)(z)=4u_{i\bar i}(z). \end{equation}
 We shall identify $\mathbb{R}^{2n}=\mathbb{R}^n\oplus \sqrt{-1} \mathbb{R}^n=\mathbb{C}^n$ in the usual way. We shall only be interested in the pairs of  two directions in $\mathbb{R}^{2n}$ which are  contained in the complex lines in $\mathbb{C}^n$. Namely, for any unit vector $v\in \mathbb{R}^{2n}$ we consider the pair $(v, Jv)$, where $J$ is the standard complex structure on $\mathbb{R}^{2n}$ and denote
 \[
  \gamma=\frac{1}{2}(v+\sqrt{-1}Jv). 
 \]
As in \eqref{E-1}, we have the following almost everywhere convergence, 
 \[
 4 u_{\gamma \bar \gamma }=\lim_{\epsilon\rightarrow 0} (w^\epsilon_{v}+w^{\epsilon}_{Jv})
 \]

First we want to show that $u_{vv}$ is uniformly bounded interior for any unit vector $v$, independent of $\epsilon$; namely $u\in C^{1, 1}$.
Since $\log \det $ is a concave function in the space of positive definite Hermitian matrices, we have (see \cite{Tru84} for example), 
\begin{equation}\label{E-0}
\sum_{i,j} u^{i\bar j}(z)  {{\partial^{2}}\over {\partial z_{i}  \partial \bar z_{j}}} w_v^{\epsilon} (z) \geq D_{2,\epsilon, v}  \log h (z),
\end{equation}

Denote by $B_R$ an Euclidean ball of radius $R$ such that $B_{3R}$ is contained in $0.8B^{2n}$. 
Applying Theorem 9.22 \cite{GT} to \eqref{E-0}, we obtain that, for $B_{3R}\subset 0.8B^{2n}$, we have
\[
\sup_{B_R} w^{\epsilon}_v \leq C\left(\|w^\epsilon_v\|_{L^p(B_{2R})}+R\|D_{2, \epsilon, v} \log h\|_{L^{2n}(B_{2R})}\right)
\]
For $\epsilon$ small ($\epsilon<R$ for example), we have
\[
\|w^\epsilon_v\|_{L^p(B_{2R})}\leq \|u_{vv}\|_{L^p(B_{3R})}\]
and
\[ \|D_{2, \epsilon, v} \log h\|_{L^{2n}(B_{2R})}\leq \|(\log h)_{vv}\|_{L^{2n}(B_{3R})}.
\]
By our assumption,  we know that $0<\triangle u\leq C$ ($\triangle$ is Euclidean laplacian) hence $u$ has uniform $W^{2, p}$ bound by the standard $L^p$ estimate. 
It follows that 
\[
\sup_{B_R}w^\epsilon_{v}\leq C.
\]
Since the above estimate is independent of $\epsilon$ and $w^\epsilon_v$ converges to $u_{vv}$ weakly in any $L^p$, it follows that 
\[
\sup_{B_R}u_{vv}\leq C.
\]
Now for a pair of unit vectors $\{v, Jv\}$ and corresponding $\gamma$, we know that
\[
0<u_{\gamma\bar \gamma}\leq C,
\]
it follows that 
\[
u_{vv}\geq -u_{Jv Jv}\geq -C. 
\]
Apparently, this $C^{1,1}$ bounds holds for interior of $0.8 B^{2n}.\;$ 
The following argument is rather standard (see \cite{Tru84}, Theorem 3.1). 
Following  Gilbarg-Trudinger \cite{GT} Section 17.4,   we set
\[
M_{s, v}^\epsilon=\displaystyle \sup_{B_{sR}} w^\epsilon_v, \; m_{s, v}^\epsilon = \displaystyle \inf_{B_{sR}} \; w_v^\epsilon\]
and correspondingly, 
\[M_{s, \gamma}^\epsilon = \displaystyle \sup_{B_{sR}} \; w_\gamma^\epsilon,\qquad {\rm and}\; m_{s, \gamma}^\epsilon =  \displaystyle \inf_{B_{sR}} \; w_\gamma^\epsilon
\]
where $s=1,2, 3$.
Similarly, set 
\[M_{s, v} = \displaystyle \sup_{B_{sR}} \; w_v,\qquad {\rm and}\; m_{s, v} = \displaystyle \inf_{B_{sR}} \; w_v.
\]
and
 \[
M_{s,\gamma} = \displaystyle \sup_{B_{sR}} \; w_\gamma,\qquad {\rm and}\; m_{s,\gamma} = \displaystyle \inf_{B_{sR}} \; w_\gamma.
\]
Here the supremum  and infimum are understood as the essential supremum and essential infimum. For $\epsilon$ small,
we have in $B_{2R}$ almost everywhere
\[
M_{3, v}-w^\epsilon_v\geq 0.
\]
It follows that $M_{3, \gamma}-w^\epsilon_\gamma$ is nonnegative almost everywhere in $B_{2R}$. 
Applying Theorem 9.22 in \cite{GT} to $M_{3, \gamma} - w^\epsilon_\gamma$ in $B_{2R}$, we  obtain
\begin{equation}\label{E-2}
\left(R^{-n} \displaystyle \int_{B_{R}}\; (M_{3, \gamma} -w^\epsilon_\gamma)^p\right)^{1\over p} \leq C_2\cdot (M_{3, \gamma}- M_{1,\gamma}^\epsilon + R \|h D_{2,\epsilon, \gamma}  \log h\|_{L^{n}(B_{2R})}).
\end{equation}
 Note that $w^\epsilon_\gamma$ converges to $w_\gamma$ almost everywhere and 
\[
M_{1, \gamma}\leq \limsup M^\epsilon_{1, \gamma}. \]
Moreover, for $\epsilon<R$, we have
\[
 \|h D_{2,\epsilon, \gamma}  \log h\|_{L^{n}(B_{2R})}\leq |h|\|h_{\gamma \bar \gamma}\|_{L^n(B_{3R})}\leq C_3R.
\]

By Lebesgue's dominated convergence theorem( in $B_{2R}$), we have, by letting $\epsilon \rightarrow 0$ in \eqref{E-2},

\[\begin{array}{lcl}
\left(R^{-n} \displaystyle \int_{B_{R}}\; (M_{3,\gamma} -w_\gamma)^p\right)^{1\over p}&  \leq & C_2\cdot (M_{3,\gamma} - M_{1, \gamma} + C_3R^2)\\ & \leq & C_4 (M_{3,\gamma} - M_{1,\gamma} + R^2). \end{array}\]

As in Section 17.4  \cite{GT} ( Lemma 17.13), we need a relation of between \emph{pure} second derivative of $u$ and its Hessian in terms of the following matrix result; such a result is only stated in real case but it extends to the Hermitian case ( for example see Section 4.3 \cite{Siu} for the Hermitian case) that, 
there exists $N$ unit  vectors  $\gamma_k, k=1,2\cdots N$ in in $\mathbb{C}^n$, including the set of coordinate unit vectors, so that we can write \[
\left( u^{i\bar j}\right)(x) = \displaystyle \sum_{k=1}^N\; \beta_k \gamma_k \otimes \bar \gamma_k
\]
with $\beta_k$ has a uniform positive lower and upper bound, depending only on $n$ and the lower and upper bound of $(u^{i\bar j})$. Using again the convexity of $\log \det $, we obtain
\[
\sum_{i,j}u^{i\bar j}(y) \left( u_{i\bar j}(x) - u_{i\bar j}(y) \right) \geq -\lambda (u(x) - u(y)).
\]
Set $w_k = u_{\gamma_k \bar \gamma_k}.\;$ Then, we have
\[
 \displaystyle \sum_{k=1}^N\; \beta_k (w_k(x) -w_k(y)) \geq -C_4 |x-y|.
\]
For $s=1, 2, 3$ set
  \[
M_{sk} = \sup_{B_{sR}} \; w_k,\qquad {\rm and}\; m_{sk} = \inf_{B_{sR}} \; w_k.
\]
Then,  we have
\[
\begin{array}{lcl} w_k(y) - w_k(x) & \leq & C_4 |y-x| - \displaystyle \sum_{i\neq k}^N\; \beta_i (w_i(y) -w_i(x)) \\
& \leq &  C_4 |y-x|  + \displaystyle \sum_{i\neq k}^N\; \beta_i (M_{3i} -w_i(y)) \\
\end{array}
\]
Evaluate $x$ at the infimum point in $B_{3R}$, we have
\[
0 < w_k(y) - m_{3k} \leq  4C_4 R  + \displaystyle \sum_{i\neq k}^N\; \beta_i (M_{3i} -w_i(y))
\]
Thus,
\[
\begin{array}{lcl} \left(R^{-n} \displaystyle \int_{B_{R}}\; (w_k(y) -m_{3k})^p\right)^{1\over p}
&\leq & 4C_4R + \displaystyle C_{5}\sum_{i\neq k}^N\;\left(R^{-n} \displaystyle \int_{B_{R}}\; (M_{3i} -w_i)^p\right)^{1\over p} \\ & \leq & 4C_4R +  \displaystyle \sum_{i\neq k}^N\; C_3 (M_{3i} - M_{1i} + R^2)
\end{array}
\]
On the other hand, by previous discussion we have
\[
\left(R^{-n} \displaystyle \int_{B_{R}}\; (M_{3k} -w_k)^p\right)^{1\over p} \leq  C_3 (M_{3k} - M_{1k} + R^2).
\]
Together we obtain
\[
M_{3k} - m_{3k} \leq 4C_4R +  \displaystyle \sum_{i=1}^N\; C_3 (M_{3i} - M_{1i} + R^2)
\]
For $s=1,2, 3$ define
\[
\omega(sR) = \displaystyle \sum_{i=1}^N\; osc (w_i) =   \displaystyle \sum_{i=1}^N\; (M_{si} - m_{si}).
\]
Then
\[
\omega(3R) \leq 4C_4R+ C_3NR^2 + C_3(\omega(3R)- \omega(R)).
\]
So
\[
\omega(R) \leq \zeta \omega (3R) + C R + C R^2.
\]
with $0 < \zeta < 1.\;$ By Lemma 8.23 in \cite{GT}, this proves the desired $C^\alpha$ estimate on $w_k$.
\end{proof}

Department of Mathematics, Stony Brook University, U.S.A. 

University of Science and Technology of China, P.R.C.

Email: xiu@math.sunysb.edu.\\

Department of Mathematics, Imperial College London, U.K.

Email: s.donaldson@imperial.ac.uk; s.sun@imperial.ac.uk.\\

\end{document}